\documentclass[11pt,reqno]{amsart}
\usepackage{mathtools} % for the command \prescript
\usepackage{a4wide}
\usepackage{dsfont} % used in \I
\usepackage{mathrsfs}
\usepackage[all]{xy}
\usepackage{graphicx} % for \scalebox in \op and for \resizebox
\usepackage[hidelinks]{hyperref}

\newtheorem{proposition}{Proposition}
\newtheorem{theorem}[proposition]{Theorem}
\newtheorem{lemma}[proposition]{Lemma}
\newtheorem{corollary}[proposition]{Corollary}
\theoremstyle{definition}
\newtheorem{definition}[proposition]{Definition}
\newtheorem{remark}[proposition]{Remark}
\newtheorem{example}[proposition]{Example}

\numberwithin{proposition}{section}
\numberwithin{equation}{section}

\newcommand{\cst}{\ensuremath{\mathrm{C}^*}}

\newcommand{\id}{\mathrm{id}}
\newcommand{\ph}{\varphi}
\newcommand{\I}{\mathds{1}}
\newcommand{\comp}{\circ}
\newcommand{\tens}{\otimes}
\newcommand{\op}{\text{\raisebox{1.5pt}{\scalebox{0.7}{\tiny{\rm{op}}}}}}
\newcommand{\is}[2]{{\left\langle{#1}\,\vline\,#2\right\rangle}}
\newcommand{\bigis}[2]{{\bigl\langle{#1}\,\bigr|\bigl.\,#2\bigr\rangle}}

\newcommand{\bigket}[1]{{\bigl|#1\bigr\rangle}}
\newcommand{\bigbra}[1]{{\bigl\langle#1\bigr|}}

\newcommand{\CC}{\mathbb{C}}

\newcommand{\MM}{\mathbb{M}}

\newcommand{\OO}{\mathbb{O}}
\newcommand{\PP}{\mathbb{P}}
\newcommand{\RR}{\mathbb{R}}
\renewcommand{\SS}{\mathbb{S}}

\newcommand{\XX}{\mathbb{X}}
\newcommand{\YY}{\mathbb{Y}}
\newcommand{\ZZ}{\mathbb{Z}}

\newcommand{\sA}{\mathsf{A}}
\newcommand{\sB}{\mathsf{B}}
\newcommand{\sC}{\mathsf{C}}
\newcommand{\sD}{\mathsf{D}}
\newcommand{\sH}{\mathsf{H}}
\newcommand{\sJ}{\mathsf{J}}

\newcommand{\sL}{\mathsf{L}}

\newcommand{\sS}{\mathsf{S}}

\newcommand{\cK}{\mathcal{K}}

\newcommand{\cU}{\mathcal{U}}

\newcommand{\bPhi}{\boldsymbol{\Phi}}
\newcommand{\bph}{\boldsymbol{\varphi}}
\newcommand{\bLambda}{\boldsymbol{\Lambda}}

\DeclareMathOperator{\B}{B}
\DeclareMathOperator{\C}{C}
\DeclareMathOperator{\M}{M}

\DeclareMathOperator{\Mat}{\mathsf{Mat}}
\DeclareMathOperator{\Mor}{Mor}

\newcommand{\alg}{\text{\tiny{\rm{alg}}}}
\newcommand{\cc}{\text{\tiny{\rm{c}}}}
\newcommand{\env}{\text{\tiny{\rm{env}}}}
\newcommand{\fp}{\boldsymbol{*}}
\newcommand{\mmax}{\text{\tiny{\rm{max}}}}

\newcommand{\QSf}{\mathfrak{QS}_{\text{\tiny{\rm{fin}}}}}
\newcommand{\QSc}{\mathfrak{QS}_{\text{\tiny{\rm{cpt}}}}}
\newcommand{\tbullet}{\text{\tiny$\bullet$}}

\title{Quantum correlations on quantum spaces}

\dedicatory{Dedicated to S.L.~Woronowicz on occasion of his eightieth birthday}

\author{Arkadiusz Bochniak}
\address{Institute of Theoretical Physics, Jagiellonian University, Poland}
\email{arkadiusz.bochniak@doctoral.uj.edu.pl}

\author{Pawe{\l} Kasprzak}
\address{Department of Mathematical Methods in Physics, Faculty of Physics, University of Warsaw, Poland}
\email{pawel.kasprzak@fuw.edu.pl}

\author{Piotr M.~Sołtan}
\address{Department of Mathematical Methods in Physics, Faculty of Physics, University of Warsaw, Poland}
\email{piotr.soltan@fuw.edu.pl}


\begin{document}

\begin{abstract}
For given quantum (non-commutative) spaces $\PP$ and $\OO$ we study the quantum space of maps $\MM_{\PP,\OO}$ from $\PP$ to $\OO$. In case of finite quantum spaces these objects turn out to be behind a large class of maps which generalize the classical $\mathrm{qc}$-correlations known from quantum information theory to the setting of quantum input and output sets. We prove various operator algebraic properties of the \cst-algebras $\C(\MM_{\PP,\OO})$ such as the lifting property and residual finite dimensionality. Inside $\C(\MM_{\PP,\OO})$ we construct a universal operator system $\SS_{\PP,\OO}$ related to $\PP$ and $\OO$ and show, among other things, that the embedding $\SS_{\PP,\OO}\subset\C(\MM_{\PP,\OO})$ is hyperrigid and has another interesting property which we call the \emph{strong extension property}. Furthermore $\C(\MM_{\PP,\OO})$ is the \cst-envelope of $\SS_{\PP,\OO}$ and a large class of non-signalling correlations on the quantum sets $\PP$ and $\OO$ arise from states on $\C(\MM_{\PP,\OO})\tens_\mmax\C(\MM_{\PP,\OO})$ as well as states on the commuting tensor product $\SS_{\PP,\OO}\tens_\cc\SS_{\PP,\OO}$. Finally we introduce and study the notion of a synchronous correlation with quantum input and output sets, prove several characterizations of such correlations and their relation to traces on $\C(\MM_{\PP,\OO})$.
\end{abstract}

\maketitle

%\tableofcontents

\allowdisplaybreaks

\section{Introduction}\label{sect:intro}

The main objects of this paper are quantum spaces of maps and quantum correlations related to them. In order to properly introduce the reader to the subject we begin with the precise definition of a quantum space as well as the associated notation.

\begin{definition}
A \emph{quantum space} is an object $\XX$ of the category dual to the category of \cst-algebras.
\end{definition}

The category of \cst-algebras is the category whose objects are \cst-algebras and whose morphisms are ``morphisms of \cst-algebras'' described in \cite[Section 1]{pseudo}, \cite[Section 0]{unbo} (see also \cite[chapter 2]{lance}), i.e.~non-degenerate $*$-homomorphisms into multiplier algebras. Given any \cst-algebras $\sA$ and $\sB$ the symbol $\Mor(\sA,\sB)$ will denote the set of all morphisms from $\sA$ to $\sB$. Note that if $\sA$ and $\sB$ are unital then a morphism from $\sA$ to $\sB$ is simply a unital $*$-homomorphism from $\sA$ to $\sB$.

According to the conventions adapted e.g.~in the theory of quantum groups (cf.~\cite{SymCoact,KdC2017,clifford,invert,synchr}) the correspondence between quantum spaces and \cst-algebras is expressed by denoting \cst-algebras by $\C(\XX)$ or $\C_0(\XX)$ where $\XX$ is the corresponding quantum space. The distinction between $\C(\cdot)$ and $\C_0(\cdot)$ is based on whether the \cst-algebra is unital (in the former case) or not (in the latter case). This notation emphasizes $\XX$ as the primary focus and this may become cumbersome. Thus in some instances this notation is dropped in favor of more traditional symbols such as $\sA$ or $\sB$ denoting \cst-algebras.

The standard terminology related to quantum spaces includes the following:
\begin{itemize}
\item a quantum space $\XX$ is called \emph{compact} if the corresponding \cst-algebra is unital (and hence denoted by $\C(\XX)$),
\item a quantum space $\XX$ is called \emph{finite} if $\C(\XX)$ is finite-dimensional.
\end{itemize}

Our aim is to generalize and study the notion of a \emph{quantum correlation} on finite sets (see below) to analogous notion for finite quantum spaces. First steps towards such a generalization have been made in e.g.~\cite{DuanWinter,BGH,TT} and we propose a more systematic, more general, and ``coordinate free'' approach thoroughly grounded in non-commutative topology and geometry. As a result we are not only able to reproduce some key results known for particular types of quantum spaces with seemingly simpler and more transparent proofs, but we can also describe and exploit the deep connection between quantum correlations and quantum spaces of maps discussed already in \cite{synchr}. Our investigations lead us to study certain universal operator systems (whose particular examples were already explored in \cite{BGH} and \cite{TT}) as well as to results about quantum spaces of maps whose generality goes far beyond the setting of finite quantum spaces. Our results seem to be related to forthcoming work of M.~Brannan, S.~Harris, I.~Todorov and L.~Turowska \cite{BHTT}.

Let us briefly discuss the notion of a \emph{quantum correlation} which was the starting point of our work. First let $P$ and $O$ be finite sets. By a \emph{quantum correlation} or a \emph{quantum strategy} on $P$ and $O$ we mean a collection of non-negative numbers $\bigl\{p(a,b|x,y)\,\bigr|\bigl.\,a,b\in{O},\:x,y\in{P}\bigr\}$ such that for each $(x,y)$ the maps
\begin{equation}\label{eq:pdP}
a\longmapsto\sum\limits_bp(a,b|x,y)\quad\text{and}\quad
b\longmapsto\sum\limits_ap(a,b|x,y)
\end{equation}
are probability distributions on $P$. A very convenient interpretation of these objects is provided by the theory of \emph{non-local games} where two players -- Alice and Bob -- are asked questions $x$ and $y$ from the set $P$ and are supposed to provide answers $a$ and $b$ from the set $O$. The number $p(a,b|x,y)$ is interpreted as the conditional probability of Alice answering $a$ and Bob answering $b$ given questions $x$ and $y$.\footnote{We are in fact dealing with a simplified version of the theory in which the sets of possible questions and answers are the same for Alice and Bob. The additional generality of allowing different question and answers sets for the two players is not essential, especially since we are aiming to study \emph{synchronous correlations}, cf.~Section \ref{sect:synchr}.} There are various ways to obtain correlations $\bigl\{p(a,b|x,y)\,\bigr|\bigl.\,a,b\in{O},\:x,y\in{P}\bigr\}$ and correlations arising from various constructions are classified into ``local'', ``quantum'', ``quantum commuting'' etc. One fundamental class of correlations is the so-called non-signalling ones. They are the ones for which the maps \eqref{eq:pdP} are independent of $y$ and $x$ respectively (they are the so-called \emph{marginals} of the correlation $\bigl\{p(a,b|x,y)\bigr\}$). For the details of the theory we refer the reader e.g.~to \cite{DuanWinter,DykemaPaulsen,HMPS,PSSTW}.

The class of correlations we will generalize to the situation when the sets $P$ and $O$ are replaced by finite quantum spaces $\PP$ and $\OO$ is the \emph{quantum commuting} (or simply \emph{$\mathrm{qc}$-cor\-re\-la\-tions}), i.e.~those of the form
\[
p(a,b|x,y)=\bigis{\xi}{E_{x,a}F_{y,b}\xi},\qquad{x,y}\in{P},\:a,b\in{O}
\]
where $\Psi$ is a unit vector in a Hilbert space $\sH$ and
\[
\bigl\{E_{x,a}\,\bigr|\bigl.\,x\in{P},\:a\in{O}\bigr\}\quad\text{and}\quad\bigl\{F_{y,b}\,\bigr|\bigl.\,y\in{P},\:b\in{O}\bigr\}
\]
are families of projections in $\B(\sH)$ such that
\begin{itemize}
\item for all $(x,y,a,b)\in{P}\times{P}\times{O}\times{O}$ we have $E_{x,a}F_{y,b}=F_{y,b}E_{x,a}$,
\item for all $x\in{P}$ we have $\sum\limits_{a}E_{x,a}=\I_\sH$ and for all $y\in{P}$ we have $\sum\limits_{b}F_{y,b}=\I_\sH$.
\end{itemize}
Note that such correlations are automatically non-signalling. Our generalization of $qc$-cor\-re\-la\-tions is explained in Section \ref{sect:realizable} where we use the terminology of \emph{realizable correlations} in order to emphasize that such correlations can be \emph{realized} in a certain well understood way.

In recent preprints \cite{BGH,TT} the task to generalize the notion of a quantum correlation to quantum sets of questions and answers has been carried out in certain special cases. Our work is very much inspired by these developments and we feel that our approach provides a more flexible framework for this theory to be developed further. Although many of our results have their origins in \cite{BGH,TT}, our treatment makes it possible to generalize many of them (sometimes quite substantially) and to differentiate between results and necessary techniques which are brought on by the particular examples of quantum spaces and general theorems and arguments independent of the particular quantum sets under investigation.

To conclude the introduction let us briefly summarize the structure and contents of the paper. We begin with the definition of the quantum space $\MM_{\PP,\OO}$ of all maps from a quantum space $\PP$ to a quantum space $\OO$ and a short discussion of conditions ensuring its existence. Then we study properties of $\MM_{\PP,\OO}$ as a function of $\PP$ and $\OO$, particularly the behavior under morphisms in either variable. We also prove that the \cst-algebra $\C(\MM_{\PP,\OO})$ is residually finite-dimensional whenever $\PP$ is finite and $\OO$ is compact and such that $\C(\OO)$ is residually finite dimensional. Next, in Section \ref{sect:SPO} we define and study a certain operator system $\SS_{\PP,\OO}$ related to quantum spaces $\PP$ and $\OO$ which later plays an important role in the description of realizable correlations. Among other things we prove that $\C(\MM_{\PP,\OO})$ is the \cst-envelope of $\SS_{\PP,\OO}$ and that the embedding $\SS_{\PP,\OO}\subset\C(\MM_{\PP,\OO})$ is hyperrigid. Moreover we show that any u.c.p.~map from $\SS_{\PP.\OO}$ to a \cst-algebra $\sB$ extends to a u.c.p.~map $\C(\MM_{\PP,\OO})\to\sB$ (a property we call the \emph{strong extension property}, see Definition \ref{def:hypersep}\eqref{def:hypersep2}). Remarkably both hyperrigidity and the strong extension property provide a conceptual explanation of some other facts concerning $\SS_{\PP,\OO}\subset\C(\MM_{\PP,\OO})$ and they were not emphasized in the case of $\PP$ and $\OO$ classical and the case of $\PP$ matricial ($\C(\PP)=\Mat_n(\CC)$) and $\OO$ classical. Section \ref{sect:LP} deals with the lifting properties of $\C(\MM_{\PP,\OO})$ and $\SS_{\PP,\OO}$ which we show to hold whenever $\C(\OO)$ has the lifting property and is separable (although $\C(\MM_{\PP,\OO})$ is almost never nuclear). In the following section we specify both $\PP$ and $\OO$ to be finite and define realizable correlations with quantum input set $\PP$ and output set $\OO$. Moreover we show that such correlations arise from states on $\C(\MM_{\PP,\OO})\tens_\mmax\C(\MM_{\PP,\OO})$ as well as states on the commuting tensor product $\SS_{\PP,\OO}\tens_\cc\SS_{\PP,\OO}$. Finally in Section \ref{sect:synchr} we study the quantum analog of synchronicity of correlations on finite quantum spaces explaining incidentally that the language of finite quantum spaces proves very convenient and effective for the generalization of this concept. We introduce an algebraic definition of synchronicity along the lines of \cite[Section 2]{BGH} and prove their characterization via traces on $\C(\MM_{\PP,\OO})$ as well as several other results analogous to those of \cite{BGH,HMPS,PSSTW,TT}.

\section{Quantum spaces of maps}

By definition mappings between quantum spaces are morphisms of the corresponding \cst-algebras understood as going in the opposite direction. Thus given quantum spaces $\XX$ and $\YY$ we identify the set $\Mor(\C_0(\YY),\C_0(\XX))$ with the set of maps from $\XX$ to $\YY$. However, the theory of quantum spaces allows ``quantum sets'' of maps which are more precisely defined as \emph{quantum families of maps} (\cite{qfam,pseudo}).

\begin{definition}\label{def:qfam}
Let $\PP$, $\OO$ and $\XX$ be quantum spaces. A \emph{quantum family of maps} from $\PP$ to $\OO$ indexed by $\XX$ is a morphism $\Phi\in\Mor(\C_0(\OO),\C_0(\PP)\tens\C_0(\XX))$.\footnote{The tensor product $\C_0(\XX)\tens\C_0(\XX)$ is the minimal tensor product of \cst-algebras. We will use other tensor products in Sections \ref{sect:SPO}, \ref{sect:realizable} and \ref{sect:synchr}.}
\end{definition}

Quantum families of maps are very general objects and hence the interesting ones are those which possess additional properties. Throughout this paper the most important role will be played by the \emph{quantum families of all maps}.

\begin{definition}
Let $\PP$ and $\OO$ be quantum spaces. We say that
\[
\bPhi_{\PP,\OO}\in\Mor\bigl(\C_0(\OO),\C_0(\PP)\tens\C_0(\MM_{\PP,\OO})\bigr)
\]
is the \emph{quantum family of all maps} from $\PP$ to $\OO$ if for any quantum space $\XX$ and any quantum family $\Psi\in\Mor(\C_0(\OO),\C_0(\PP)\tens\C_0(\XX))$ there exists a unique $\Lambda\in\Mor(\C_0(\MM_{\PP,\OO}),\C_0(\XX))$ such that the diagram
\[
\xymatrix{
\C_0(\OO)\ar[rr]^-{\bPhi_{\PP,\OO}}\ar@{=}[d]&&\C_0(\PP)\tens\C_0(\MM_{\PP,\OO})\ar[d]^{\id\tens\Lambda}\\
\C_0(\OO)\ar[rr]^-{\Psi}&&\C_0(\PP)\tens\C_0(\XX)
}
\]
is commutative.\footnote{This is a diagram in the category of \cst-algebras, so the arrows represent mappings into the corresponding multiplier algebras and their compositions require certain maneuvers, see \cite[Section 1]{pseudo} or \cite[Chapter 2]{lance}.}
\end{definition}

It is easy to see that if $(\MM_{\PP,\OO},\bPhi_{\PP,\OO})$ exists for the given $\PP$ and $\OO$ then it is unique up to a natural notion of isomorphism, i.e.~if $(\MM,\bPhi)$ is another pair with the same universal property then there is an isomorphism $\bLambda\colon\C_0(\MM_{\PP,\OO})\to\C_0(\MM)$ such that $\bPhi=(\id\tens\bLambda)\comp\bPhi_{\PP,\OO}$. The quantum space $\MM_{\PP,\OO}$ is called the \emph{quantum space of all maps} from $\PP$ to $\OO$.

\begin{theorem}\label{thm:exMPO}
Let $\PP$ be a finite quantum space and $\OO$ a compact quantum space. Then the quantum space $\MM_{\PP,\OO}$ of all maps from $\PP$ to $\OO$ exists and is compact. Moreover the \cst-algebra $\C(\MM_{\PP,\OO})$ is generated by the set
\[
\bigl\{(\omega\tens\id)\bPhi_{\PP,\OO}(a)\,\bigr|\bigl.\,a\in\C(\OO),\:\omega\in\C(\PP)^*\bigr\}.
\]
\end{theorem}

Theorem \ref{thm:exMPO} is proved in \cite{invert,qfam} in the case when $\C(\OO)$ is unital and finitely generated (i.e.~a quotient of the full group \cst-algebra of a finitely generated free group). The latter assumption can be dropped and the proof from \cite[Appendix]{invert} can be repeated almost verbatim. The only price one has to pay is to consider the full group \cst-algebra of the free group with possibly a very large number of generators. An alternative way to prove the existence of $(\MM_{\PP,\OO},\bPhi_{\PP,\OO})$ in the general case is indicated in Remark \ref{rem:existenceCMPO}.

Let us note that, the universal map $\bPhi_{\PP,\OO}$ is injective. Indeed, one can consider $\XX=\OO$ and
\[
\Psi\colon\C(\OO)\ni{a}\longmapsto\I\tens{a}\in\C(\PP)\tens\C(\OO).
\]
Then there exists $\Lambda\colon\C(\MM_{\PP,\OO})\to\C(\OO)$ such that
\[
\I\tens{a}=(\id\tens\Lambda)\bigl(\bPhi_{\PP,\OO}(a)\bigr),\qquad{a}\in\C(\OO)
\]
which shows that $\bPhi_{\PP,\OO}(a)\neq{0}$ if $a\neq{0}$.

\begin{example}\label{onepoint}
Let $\OO$ be a compact quantum space and let $\boldsymbol{1}$ denote the one-point space (so $\C(\boldsymbol{1})=\CC$). Then one easily finds that $\MM_{\boldsymbol{1},\OO}=\OO$ in the sense that the mapping
\[
\C(\OO)\ni{x}\longmapsto{1}\tens{x}\in\C(\boldsymbol{1})\tens\C(\OO)
\]
has the universal property of $\bPhi_{\boldsymbol{1},\OO}$.
\end{example}

Other examples of quantum spaces of maps were studied e.g.~in \cite{qcomm,qfam,so3,houston,pseudo} as well as \cite{synchr}, where their relation to non-local games was studied. Let us note that the \cst-algebra $\mathcal{P}_{n,c}$ of \cite{BGH} is precisely $\C(\MM_{\PP,\OO})$ for $\PP$ and $\OO$ such that $\C(\PP)=\Mat_n(\CC)$ and $\C(\OO)=\CC^c$.

\subsection{Disjoint sums of quantum spaces}\hspace*{\fill}

Throughout the paper the symbol $\fp$ will denote the universal free product of unital \cst-algebras amalgamated over the units.

Consider two compact quantum spaces $\PP_1$ and $\PP_2$. The quantum space $\PP_1\sqcup\PP_2$ is defined by setting $\C(\PP_1\sqcup\PP_2)=\C(\PP_1)\oplus\C(\PP_2)$ with the inclusions of $\PP_1$ and $\PP_2$ into $\PP_1\sqcup\PP_2$ described by the projections $\mathrm{p}_i\colon\C(\PP_1\sqcup\PP_2)\to\C(\PP_i)$ for $i=1,2$.

\begin{proposition}\label{freePP}
Let $\PP_1,\PP_2$ be finite quantum spaces and $\OO$ be a compact quantum space. Then the \cst-algebra $\C(\MM_{\PP_1\sqcup\PP_2,\OO})$ is isomorphic to the universal free product $\C(\MM_{\PP_1,\OO})\fp\C(\MM_{\PP_2,\OO})$, and with this identification the universal quantum family of maps
\[
\bPhi_{\PP_1\sqcup\PP_2,\OO}\colon\C(\OO)\longrightarrow\C(\PP_1\sqcup\PP_2)\tens\C(\MM_{\PP_1\sqcup\PP_2,\OO})
\]
is given by
\[
\bPhi_{\PP_1\sqcup\PP_2,\OO}(a)=(\jmath_1\tens\imath_1)\bigl(\bPhi_{\PP_1,\OO}(a)\bigr)+(\jmath_2\tens\imath_2)\bigl(\bPhi_{\PP_2,\OO}(a)\bigr),\qquad{a}\in\C(\OO),
\]
where for $i=1,2$ the maps $\jmath_i\colon\C(\PP_i)\to\C(\PP_1\sqcup\PP_2)$ are the (non-unital) inclusions of direct summands and $\imath_i\colon\C(\MM_{\PP_i,\OO})\to\C(\MM_{\PP_1\sqcup\PP_2,\OO})$ are the inclusions onto the copies of $\C(\MM_{\PP_i,\OO})$ in the free product.
\end{proposition}

\begin{proof}
The reasoning is similar to the proof of \cite[Theorem 2.1]{houston}: it is enough to show that the pair $(\C(\MM),\bPhi)$ with
\[
\C(\MM)=\C(\MM_{\PP_1,\OO})\fp\C(\MM_{\PP_2,\OO})
\]
and $\bPhi$ defined by
\begin{equation}\label{defbPhi}
\bPhi(a)=(\jmath_1\tens\imath_1)\bigl(\bPhi_{\PP_1,\OO}(a)\bigr)+(\jmath_2\tens\imath_2)\bigl(\bPhi_{\PP_2,\OO}(a)\bigr),\qquad{a}\in\C(\OO)
\end{equation}
has the universal property of $(\MM_{\PP_1\sqcup\PP_2,\OO},\bPhi_{\PP_1\sqcup\PP_2,\OO})$. The details are left to the reader.
\end{proof}

\begin{remark}\label{remi2}
Let $\OO,\PP_1$ and $\PP_2$ be as above and $\mathrm{p}_1\colon\C(\PP_1)\oplus\C(\PP_2)\to \C(\PP_1)$ the canonical projection. By composing $\bPhi_{\PP_1\sqcup\PP_2,\OO}$ with $(\mathrm{p}_1\tens\id)$ we obtain a unital $*$-homomorphism $\C(\OO)\to\C(\PP_1)\tens\C(\MM_{\PP_1\sqcup\PP_2,\OO})$, so it must be of the form $(\id\tens\Lambda)\comp\bPhi_{\PP_1,\OO}$ for a unique $\Lambda\colon\C(\MM_{\PP_1,\OO})\to\C(\MM_{\PP_1\sqcup\PP_2,\OO})$. In view of the above proposition (particularly the uniqueness of $\Lambda$) we easily see that $\Lambda=\imath_1$.
\end{remark}

Using Proposition \ref{freePP} we can give a rather concrete description of all \cst-algebras $\C(\MM_{\PP,\OO})$. To this end we need the following lemma:

\begin{lemma}[{\cite[Proposition 2.18]{pisier}}]\label{lem:Pisier}
Let $\sC$ be a unital \cst-algebra and let $\gamma\colon\Mat_n(\CC)\to\sC$ be a unital $*$-homomorphism. Let
\[
\sD=\bigl\{c\in\sC\,\bigr|\bigl.\,c\gamma(x)=\gamma(x)c\text{ for all }x\in\Mat_n(\CC)\bigr\}.
\]
Then $\sD$ is a unital \cst-algebra and $\sC$ is isomorphic to $\Mat_n(\CC)\tens\sD$.
\end{lemma}

\begin{proof}
The subset $\sD$ is clearly a unital \cst-subalgebra of $\sC$. Furthermore the mapping
\[
x\tens{d}\longmapsto\gamma(x)d,\qquad{x}\in\Mat_n(\CC),\:d\in\sD
\]
extends to a unital $*$-homomorphism $\Gamma\colon\Mat_n(\CC)\tens\sD\to\sC$. To see that it is injective, for $i,j\in\{1,\dotsc,n\}$ let $E_{i,j}=\gamma(e_{i,j})$, where $\{e_{i,j}\}$ are the matrix units in $\Mat_n(\CC)$. If
\begin{equation}\label{eq:eczero}
\sum_{i,j}e_{i,j}\tens{c_{i,j}}
\end{equation}
belongs to the kernel of $\Gamma$ then for any $k,l,r$
\[
0=\Gamma\Biggl((e_{k,l}\tens\I)\biggl(\sum_{i,j}e_{i,j}\tens{c_{i,j}}\biggr)(e_{k,r}\tens\I)\Biggr)=E_{k,k}c_{l,r}.
\]
Hence $0=\sum\limits_kE_{k,k}c_{l,r}=c_{l,r}$ for all $r,l$, so that \eqref{eq:eczero} is zero.

Surjectivity of $\Gamma$ follows from the fact that if $c\in\sC$ then putting
\[
c_{i,j}=\sum_{k=1}^nE_{k,i}cE_{j,k},\qquad{i,j}=1,\dotsc,n
\]
we obtain $c_{i,j}\in\sD$ for all $i,j$ and $\Gamma\biggl(\sum\limits_{i,j}e_{i,j}\tens{c_{i,j}}\bigg)=\sum\limits_{i,j}E_{i,j}c_{i,j}=c$.
\end{proof}

\begin{proposition}\label{prop:sD}
Let $\C(\PP)=\Mat_n(\CC)$ and let $\OO$ be a compact quantum space. Then $\C(\MM_{\PP,\OO})$ is the relative commutant of $\C(\PP)$ in $\C(\PP)\fp\C(\OO)$ and $\bPhi_{\PP,\OO}$ is the composition of the inclusion $\C(\OO)\to\C(\PP)\fp\C(\OO)$ with the isomorphism $\C(\PP)\fp\C(\OO)\to\C(\PP)\tens\C(\MM_{\PP,\OO})$ described in Lemma \ref{lem:Pisier}.
\end{proposition}

\begin{proof}
Let $\XX$ be a quantum space and $\Psi\in\Mor(\C(\OO),\C(\PP)\tens\C_0(\XX))$ a quantum family of maps $\PP\to\OO$. Furthermore let $\sD$ denote the relative commutant of $\C(\PP)$ in $\C(\PP)\fp\C(\OO)$.

By the universal property of free products there exists a unital $*$-homomorphism $\Pi\colon\C(\PP)\fp\C(\OO)\to\C(\PP)\tens\M(\C_0(\XX))$ such that $\Pi(\imath_1(x))=x\tens\I$ for all $x\in\C(\PP)$ and $\Pi(\imath_2(y))=\Psi(y)$ for all $y\in\C(\OO)$. Note that $\Pi\in\Mor(\C(\PP)\fp\C(\OO),\C(\PP)\tens\C_0(\XX))$.

Now if $d$ belongs to $\sD$ then $\Pi(d)(x\tens\I)=(x\tens\I)\Pi(d)$ for all $d\in\C(\PP)$ because
\[
\Pi(d)(x\tens\I)=\Pi(d)\Pi\bigl(\imath_1(x)\bigr)=\Pi\bigl(d\imath_1(x)\bigr)=\Pi\bigl(\imath_1(x)d\bigr)=(x\tens\I)\Pi(d).
\]
Therefore the element $\Pi(d)$ belongs to the commutant of $\C(\PP)\tens\I$ in $\C(\PP)\tens\M(\C_0(\XX))$, i.e.~to $\I\tens\M(\C_0(\XX))$. As a consequence we can define a unital $*$-homomorphism $\Lambda\colon\sD\to\M(\C_0(\XX))$ by
\[
\Pi(d)=\I\tens\Lambda(d),\qquad{d}\in\sD
\]
and defining $\bPhi$ as the composition $\Gamma\comp\imath_2\colon\C(\OO)\to\C(\PP)\tens\sD$ we immediately get
\begin{equation}\label{eq:uniqDLambda}
(\id\tens\Lambda)\bigl(\bPhi(y)\bigr)=\Pi\bigl(\imath_2(y)\bigr)=\Psi(y),\qquad{y}\in\C(\OO).
\end{equation}
Since $\sD$ is obviously generated by slices of the form $(\omega\tens\id)\bPhi(y)$ ($y\in\C(\OO)$, $\omega\in\C(\PP)^*$), we immediately see that $\Lambda$ is uniquely determined by \eqref{eq:uniqDLambda}. It follows that $(\sD,\bPhi)$ has the universal property of $(\C(\MM_{\PP,\OO}),\bPhi_{\PP,\OO})$.
\end{proof}

\begin{corollary}\label{cor:DCMPO}
Let $\OO$ be a compact quantum space and $\PP$ a finite quantum space with
\[
\C(\PP)=\bigoplus_{i=1}^m\Mat_{n_i}(\CC).
\]
Then the \cst-algebra $\C(\MM_{\PP,\OO})$ is a free product $\sD_1\fp\dotsm\fp\sD_m$, where $\sD_i$ is the relative commutant of $\Mat_{n_i}(\CC)$ in $\Mat_{n_i}(\CC)\fp\C(\OO)$.
\end{corollary}

\begin{remark}\label{rem:existenceCMPO}
The explicit description of $\C(\MM_{\PP,\OO})$ and $\bPhi_{\PP,\OO}$ in the case $\C(\PP)=\Mat_n(\CC)$ given in Proposition \ref{prop:sD} together with Proposition \ref{freePP} provides an alternative way to prove existence of $\C(\MM_{\PP,\OO})$ for arbitrary compact $\OO$ and finite $\PP$.
\end{remark}

% \subsection{Residual finite dimensionality}\hspace*{\fill}

\begin{theorem}
Let $\OO$ be a compact quantum space such that $\C(\OO)$ is residually finite dimensional and let $\PP$ be a finite quantum space. Then $\C(\MM_{\PP,\OO})$ is residually finite dimensional.
\end{theorem}

\begin{proof}
By Corollary \ref{cor:DCMPO} the \cst-algebra $\C(\MM_{\PP,\OO})$ is a free product of algebras which are subalgebras of free products of the form $\Mat_n(\CC)\fp\C(\OO)$. Since residual finite dimensionality passes to free products (\cite[Theorem 3.2]{ExelLoring}) and to subalgebras, $\C(\MM_{\PP,\OO})$ is residually finite dimensional.
\end{proof}

\begin{corollary}
For any finite quantum spaces $\PP$ and $\OO$ the \cst-algebra $\C(\MM_{\PP,\OO})$ possesses a faithful trace.
\end{corollary}

Let us note that existence of traces on $\C(\MM_{\PP,\OO})$ is important in the study of synchronous correlations (see Section \ref{sect:synchr}).

\subsection{Functorial properties}\label{FuncorialProperties}\hspace*{\fill}

In this section we will study the properties of the assignment
\begin{equation}\label{eq:POtoMPO}
(\PP,\OO)\longmapsto\MM_{\PP,\OO},
\end{equation}
where $\PP$ is a finite quantum space and $\OO$ is a compact quantum space. For brevity let us denote by $\QSf$ and $\QSc$ the full subcategories of the category of quantum spaces consisting of the finite and compact quantum spaces respectively.

The proof the next proposition easily follows from the universal property of $\MM_{\PP,\OO}$. 
\begin{proposition}
The mapping \eqref{eq:POtoMPO} is a bi-functor $\QSf\times\QSc\to\QSc$. Given $\PP_1,\PP_2\in\operatorname{Ob}(\QSf)$, $\OO_1,\OO_2\in\operatorname{Ob}(\QSc)$ and 
\[
\rho\colon\C(\PP_2)\longrightarrow\C(\PP_1),\quad\pi\colon\C(\OO_1)\longrightarrow\C(\OO_2),
\]
the associated map $\MM_{\rho,\pi}\colon\C(\MM_{\PP_1,\OO_1})\to\C(\MM_{\PP_2,\OO_2})$ is the unique $\Lambda$ making the diagram
\[
\xymatrix{
\C(\OO_1)\ar@{=}[d]\ar[rrr]^-{\bPhi_{\PP_1,\OO_1}}&&&\C(\PP_1)\tens\C(\MM_{\PP_1,\OO_1})\ar[d]^{\id\tens\Lambda}
\\
\C(\OO_1)\ar[rrr]^-{(\rho\tens\id)\comp\bPhi_{\PP_2,\OO_2}\comp\pi}&&&\C(\PP_1)\tens\C(\MM_{\PP_2,\OO_2})
}
\]
commutative. In particular the functor $\MM_{\tbullet,\tbullet}$ is contravariant with respect to the first variable and covariant with respect to the second one.
\end{proposition}

\begin{remark}\label{rem:exiso}
It is an obvious consequence of the functoriality of $\MM_{\tbullet,\tbullet}$ that if $\PP$ and $\PP'$ are finite quantum spaces such that $\C(\PP)\cong\C(\PP')$ then for any compact quantum space $\OO$ we have $\C(\MM_{\PP,\OO})\cong\C(\MM_{\PP',\OO})$. Similarly if $\C(\OO)\cong\C(\OO')$ for some compact quantum spaces $\OO$ and $\OO'$ then $\C(\MM_{\PP,\OO})\cong\C(\MM_{\PP,\OO'})$ for any finite quantum space $\PP$.
\end{remark}

\newcounter{c}
\begin{theorem}\label{thm:functoriality}
Let $\PP,\PP_1,\PP_2\in\operatorname{Ob}(\QSf)$ and $\OO,\OO_1,\OO_2\in\operatorname{Ob}(\QSc)$ and
\[
\rho\colon\C(\PP_2)\longrightarrow\C(\PP_1),\quad\pi\colon\C(\OO_1)\longrightarrow\C(\OO_2).
\]
Then 
\begin{enumerate}
\item\label{functoriality:1} if $\pi$ is surjective then so is $\MM_{\id,\pi}\colon\C(\MM_{\PP,\OO_1})\to \C(\MM_{\PP,\OO_2})$,
\item\label{functoriality:2} if $\rho$ is injective then $\MM_{\rho,\id}\colon\C(\MM_{\PP_1,\OO})\to \C(\MM_{\PP_2,\OO})$ is surjective.\setcounter{c}{\value{enumi}}
\end{enumerate}
Moreover,
\begin{enumerate}
\setcounter{enumi}{\value{c}}
\item\label{functoriality:3} if $\pi$ is injective then so is $\MM_{\id,\pi}\colon\C(\MM_{\PP,\OO_1})\to \C(\MM_{\PP,\OO_2})$,
\item\label{functoriality:4} if $\rho$ is surjective then $\MM_{\rho,\id}\colon\C(\MM_{\PP_1,\OO})\to \C(\MM_{\PP_2,\OO})$ is injective.
\end{enumerate}
\end{theorem}

\begin{proof} 
Ad \eqref{functoriality:1}. Since $\pi$ is surjective, we have
\begin{align*}
\bigl\{(\omega\tens\id)\bPhi_{\PP,\OO_2}(y)\,\bigr|\bigl.\,y\in\C(\OO_2),&\:\omega\in\C(\PP)^*\bigr\}\\
&=\bigl\{(\omega\tens\id)\bPhi_{\PP,\OO_2}\bigl(\pi(x)\bigr)\,\bigr|\bigl.\,x\in\C(\OO_1),\:\omega\in\C(\PP)^*\bigr\}\\
&=\MM_{\id,\pi}\Bigl(
\bigl\{(\omega\tens\id)\bPhi_{\PP,\OO_1}(x)\,\bigr|\bigl.\,x\in\C(\OO_1),\:\omega\in\C(\PP)^*\bigr\}
\Bigr).
\end{align*}
Furthermore, since $\C(\MM_{\PP,\OO_2})$ is generated by the set on the left-hand side, we get that $\MM_{\id,\pi}$ is surjective.

Ad \eqref{functoriality:2}. We have
\begin{align*}
\MM_{\rho,\id}\Bigl(\bigl\{(\omega\tens\id)\bPhi_{\PP_1,\OO}(x)\,\bigr|\bigl.\,x\in\C(\OO),&\:\omega\in\C(\PP_1)^*\bigr\}\Bigr)\\
&=\bigl\{\bigl((\omega\comp\rho)\tens\id\bigr)\bPhi_{\PP_2,\OO}(x)\,\bigr|\bigl.\,x\in\C(\OO),\:\omega\in\C(\PP_1)^*\bigr\}\\
&=\bigl\{(\widetilde{\omega}\tens\id)\bPhi_{\PP_2,\OO}(x)\,\bigr|\bigl.\,x\in\C(\OO),\:\widetilde{\omega}\in\C(\PP_2)^*\bigr\}
\end{align*}
due to the injectivity of $\rho$. As the last set generates $\C(\MM_{\PP_2,\OO})$ we see that $\MM_{\rho,\id}$ is a surjection.

Ad \eqref{functoriality:3}
Let us first prove the statement \eqref{functoriality:3} in the case $\C(\PP)=\Mat_k(\CC)$. Let us denote $\sC_1 = \Mat_k(\CC)\fp \C(\OO_1)$ and $\sC_2 =\Mat_k(\CC)\fp \C(\OO_2)$. By \cite[Theorem 4.2]{pullback} the injective $*$-homomorphism $\pi$ defined an injective $*$-homomorphism $\id\fp\pi\colon\sC_1\to\sC_2$. Using the identifications
\begin{align*}
\sC_1 &= \Mat_k(\CC)\tens\C(\MM_{\PP,\OO_1})\\ 
\sC_2 &= \Mat_k(\CC)\tens\C(\MM_{\PP,\OO_2})
\end{align*}
(cf.~Proposition \ref{prop:sD}) we easily get $(\id\fp\pi)(x) =\MM_{\id,\pi}(x)$ for all $x\in\C(\MM_{\PP,\OO_1})$, and conclude the injectivity of $\MM_{\id,\pi}$. 

In the general case
\[
\C(\PP)=\bigoplus\limits_{i=1}^m\Mat_{k_i}(\CC)
\]
(i.e.~$\PP=\PP_1\sqcup\ldots\sqcup\PP_m$ with $\C(\PP_i)=\Mat_{k_i}(\CC)$) and the proof is completed by using Proposition \ref{freePP} and observing that 
$\MM_{\id_{\PP_1\sqcup\dotsm\sqcup\PP_m},\pi}=\MM_{\id_{\PP_1},\pi}\fp\dotsm\fp\MM_{\id_{\PP_m},\pi}$. Since $\MM_{\id_{\PP_i},\pi}$ are injective for $i=1,\ldots,m$, the $*$-homomorphism $\M_{\id_{\PP_1},\pi}\fp\dotsm\fp\M_{\id_{\PP_m},\pi}$ is injective (c.f. \cite[Theorem 4.2]{pullback}).
 
Ad \eqref{functoriality:4}. Since the \cst-algebras $\C(\PP_1)$ and $\C(\PP_2)$ are finite-dimensional, the fact that $\C(\PP_2)$ surjects onto $\C(\PP_1)$ means that $\PP_2=\PP'\sqcup\PP_1$ for some finite quantum space $\PP'$. More precisely, there exists an isomorphism $\sigma\colon\C(\PP')\oplus\C(\PP_1)\to\C(\PP_2)$ such that $\rho\comp\sigma\colon\C(\PP')\oplus\C(\PP_1)\to\C(\PP_1)$ is the projection onto $\C(\PP_1)$. In view of Remark \ref{remi2}
\[
\bigl((\rho\comp\sigma)\tens\id\bigr)\comp\bPhi_{\PP'\sqcup\PP_1,\OO}=(\id\tens\imath_1)\comp\bPhi_{\PP_1,\OO},
\]
where $\imath_1$ is the inclusion of $\C(\MM_{\PP_1,\OO})$ into $\C(\MM_{\PP'\sqcup\PP_1,\OO})=\C(\MM_{\PP',\OO})\fp\C(\MM_{\PP_1,\OO})$. Now the mapping $\MM_{\sigma,\id}\colon\C(\MM_{\PP_2,\OO})\to\C(\MM_{\PP'\sqcup\PP_1,\OO})$ is an isomorphism satisfying
\[
(\sigma\tens\id)\comp\bPhi_{\PP'\sqcup\PP_1,\OO}=(\id\tens\MM_{\sigma,\id})\comp\bPhi_{\PP_2,\OO}.
\]
It follows that
\[
(\rho\tens\id)\comp\bPhi_{\PP_2,\OO}=\bigl((\rho\comp\sigma)\tens(\MM_{\sigma,\id})^{-1}\bigr)\comp\bPhi_{\PP'\sqcup\PP_1,\OO}=\Bigl(\id\tens\bigl((\MM_{\sigma,\id_\OO})^{-1}\comp\imath_1\bigr)\Bigr)\comp\bPhi_{\PP_1,\OO}
\]
and consequently $\MM_{\rho,\id}=(\MM_{\sigma,\id})^{-1}\comp\imath_1$ is injective.
\end{proof}

\subsection{The opposite algebra}\label{sectOpposite}\hspace*{\fill}

Let $\XX$ be a quantum space. We will write $\XX^\op$ for the quantum space corresponding to the opposite \cst-algebra (cf.~\cite[Section 1.3.3]{dixmier}) of $\C_0(\XX)$. Thus if, for example, $\XX$ is compact then by definition $\C(\XX^\op)=\C(\XX)^\op$.

\begin{proposition}\label{prop:anti}
Let $\PP$ be a finite quantum space and $\OO$ a compact quantum space. Then the pair $(\MM_{\PP^\op,\OO^\op},\bPhi_{\PP^\op,\OO^\op})$ is naturally isomorphic to $({\MM_{\PP,\OO}}^\op,\bPhi_{\PP,\OO})$, where the map of vector spaces $\bPhi_{\PP,\OO}\colon\C(\OO)\to\C(\PP)\tens\C(\MM_{\PP,\OO})$ is regarded as a $*$-homomorphism $\C(\OO)^\op\to\C(\PP)^\op\tens\C(\MM_{\PP,\OO})^\op$.
\end{proposition}

\begin{proof}
It is enough to show that the pair $({\MM_{\PP,\OO}}^\op,\bPhi_{\PP,\OO})$ with $\bPhi_{\PP,\OO}$ understood as a map $\C(\OO)^\op\to\C(\PP)^\op\tens\C(\MM_{\PP,\OO})^\op$ has the universal property of $(\MM_{\PP^\op,\OO^\op},\bPhi_{\PP^\op,\OO^\op})$. Thus let $\XX$ be a quantum space and let $\Psi\in\Mor(\C(\OO)^\op,\C(\PP)^\op\tens\C_0(\XX))$. Let
\[
\kappa_{\C(\OO)}\colon\C(\OO)\longrightarrow\C(\OO)^\op,\quad
\kappa_{\C(\PP)}\colon\C(\PP)\longrightarrow\C(\PP)^\op,\quad
\kappa_{\C_0(\XX)}\colon\C_0(\XX)\longrightarrow\C_0(\XX)^\op
\]
and
\[
\kappa_{\C(\MM_{\PP,\OO})}\colon\C(\MM_{\PP,\OO})\longrightarrow\C(\MM_{\PP,\OO})^\op
\]
be the identity maps of vector spaces considered as anti-isomorphisms between the respective \cst-algebras. Then
\[
({\kappa_{\C(\PP)}}^{-1}\tens\kappa_{\C_0(\XX)})\comp\Psi\comp\kappa_{\C(\OO)}=(\id\tens\Lambda)\comp\bPhi_{\PP,\OO}
\]
for a unique $\Lambda\in\Mor(\C(\MM_{\PP,\OO}),\C_0(\XX)^\op)$. In other words
\[
\Psi=\bigl(\id\tens({\kappa_{\C_0(\XX)}}^{-1}\comp\Lambda\comp{\kappa_{\C(\MM_{\PP,\OO})}}^{-1})\bigr)\comp\bigl((\kappa_{\C(\PP)}\tens\kappa_{\C(\MM_{\PP,\OO})})\comp\bPhi_{\PP,\OO}\comp{\kappa_{\C(\OO)}}^{-1}\bigr).
\]
Note that $(\kappa_{\C(\PP)}\tens\kappa_{\C(\MM_{\PP,\OO})})\comp\bPhi_{\PP,\OO}\comp{\kappa_{\C(\OO)}}^{-1}$ is precisely the map $\bPhi_{\PP,\OO}$ regarded as a $*$-homomorphism $\C(\OO)^\op\to\C(\PP)^\op\tens\C(\MM_{\PP,\OO})^\op$. It follows that $\Psi$ can be expressed as
\[
\Psi=(\id\tens\Lambda')\comp\bPhi_{\PP,\OO}
\]
(with the last map properly understood) for a unique $\Lambda'\in\Mor(\C(\MM_{\PP,\OO}),\C_0(\XX))$.
\end{proof}

Note that for any finite quantum space $\PP$ the \cst-algebra $\C(\PP)$ is isomorphic to $\C(\PP)^\op$, so by Remark \ref{rem:exiso}
\begin{equation}\label{PantiP}
\C(\MM_{\PP,\OO})\cong\C(\MM_{\PP^\op,\OO}).
\end{equation}

\begin{corollary}\label{CorOp}
Let $\OO$ be a compact quantum space such that $\C(\OO)$ is isomorphic to $\C(\OO)^\op$. Then for any finite quantum space $\PP$ the \cst-algebra $\C(\MM_{\PP,\OO})$ is isomorphic to $\C(\MM_{\PP,\OO})^\op$.
\end{corollary}

\begin{proof}
By assumption and Remark \ref{rem:exiso} we have $\C(\MM_{\PP,\OO})\cong\C(\MM_{\PP,\OO^\op})$, so in view of \eqref{PantiP} and Proposition \ref{prop:anti} we get $\C(\MM_{\PP,\OO})\cong\C(\MM_{\PP^\op,\OO^\op})\cong\C(\MM_{\PP,\OO})^\op$.
\end{proof}

Note that if $\OO$ is a finite quantum space then $\C(\OO)$ is isomorphic to $\C(\OO)^\op$, so $\C(\MM_{\PP,\OO})$ is isomorphic to $\C(\MM_{\PP,\OO})^\op$ for all finite quantum spaces $\PP$ and $\OO$. In particular Corollary \ref{CorOp} is a generalization of \cite[Lemma 1.16]{BGH}.

\section{The universal operator system}\label{sect:SPO}

In this and following sections (in particular in the proof of Theorem \ref{LPprop}) we will use the notion of a (\emph{right/left}) \emph{multiplicative domain} of a completely positive map. Our terminology will be the same as in \cite[Chapter 3]{paulsen}.

Also we will freely use the fact that a u.c.p.~map $\psi\colon\sA\to\sB$ between unital \cst-algebras defines a non-degenerate completely positive map $\id\tens\psi\colon\cK\tens\sA\to\cK\tens\sB$, where $\cK$ denotes the \cst-algebra of compact operators on $\ell_2$. This map, in turn, extends uniquely by strict continuity to a u.c.p.~map between the multiplier algebras: $\M(\cK\tens\sA)\to\M(\cK\tens\sB)$. Additionally if, for any two \cst-algebras $\sC$ and $\sD$ and any $\omega\in\sC^*$ the slice map $\omega\tens\id$ extends uniquely to a mapping $\M(\sC\tens\sD)\to\M(\sD)$. All this is contained in \cite[Chapters 5 and 8]{lance}. 

\begin{lemma}\label{mult_dom_arg_K}
Let $\sA$ and $\sB$ be unital \cst-algebras. If $\ph\colon\sA\to\sB$ is a unital completely positive map and $a\in\M(\cK\tens\sA)$ belongs to the multiplicative domain of $\id\tens\ph\colon\M(\cK\tens\sA)\to\M(\cK\tens\sB)$ then for any $\omega\in\cK^*$ the element $(\omega\tens\id)(a)$ belongs to the multiplicative domain of $\ph$.
\end{lemma}

\begin{proof}
Any $\omega\in\cK^*$ is of the form $\omega=\operatorname{Tr}(\cdot\rho)=\operatorname{Tr}(\rho\cdot)$ for some trace-class operator $\rho$. Since any such $\omega$ can be approximated in the norm of $\cK^*$ by functionals of the same form with $\rho$ finite-dimensional, and the multiplicative domain of a u.c.p.~map is closed (\cite[Exercise $4.2(\text{iii})$]{paulsen}), it is enough to prove that $(\omega\tens\id)(a)$ is in the multiplicative domain of $\ph$ for such $\omega$.

Let $r$ and $l$ be the left and right support of $\rho$, i.e.~$l$ is the projection onto the range of $\rho$ (assumed to be finite-dimensional) and $r$ is the projection onto the range of $\rho^*$ (which is also finite-dimensional). Then for any $y\in\cK$ we have
\[
\omega(yl)=\operatorname{Tr}(yl\rho)=\operatorname{Tr}(y\rho)=\omega(y)\quad\text{and}\quad
\omega(ry)=\operatorname{Tr}(ry\rho)=\operatorname{Tr}(y\rho)
\]

Since $a$ is in the multiplicative domain of $\id\tens\ph$, for any $z\in\sA$ we have
\[
(\id\tens\ph)\bigl(a(l\tens{z})\bigr)=\bigl((\id\tens\ph)(a)\bigr)\bigl((\id\tens\ph)(l\tens{z})\bigr)=\bigl((\id\tens\ph)(a)\bigr)\bigl(l\tens\ph(z)\bigr).
\]
Applying $(\omega\tens\id)$ to both sides we obtain
\[
\ph\Bigl(\bigl((\omega\tens\id)(a)\bigr)z\Bigr)=\ph\bigl((\omega\tens\id)(a)\bigr)\ph(z)
\]
Similarly
\[
(\id\tens\ph)\bigl((r\tens{z})a\bigr)=\bigl((\id\tens\ph)(r\tens{z})\bigr)\bigl((\id\tens\ph)(a)\bigr)=\bigl(r\tens\ph(z)\bigr)\bigl((\id\tens\ph)(a)\bigr).
\]
and application of $(\omega\tens\id)$ to both sides yields
\[
\ph\Bigl(z\bigl((\omega\tens\id)(a)\bigr)\Bigr)=\ph(z)\ph\bigl((\omega\tens\id)(a)\bigr).
\]
Since this is true for all $z$, the element $(\omega\tens\id)(a)$ belongs to the multiplicative domain of $\ph$.
\end{proof}

From this section onward we will make frequent use of Kasparov's dilation theorem (\cite[Theorem 9]{jungeetal}, cf.~also \cite[Theorem 6.5]{lance}). One of its direct consequences is the following lemma:

\begin{lemma}\label{aux_ksp}
Let $\PP$ be a finite quantum space and $\OO$ a compact quantum space such that $\C(\OO)$ is separable. Let $\sB$ be a separable unital \cst-algebra and let $\psi\colon\C(\OO)\to\C(\PP)\tens\sB$ be a u.c.p.~map. Then there exists $\Psi\in\Mor(\C(\OO),\C(\PP)\tens\cK\tens\sB)$ such that
\[
\psi(x)=(\id\tens\omega_{1,1}\tens\id)\bigl(\Psi(x)\bigr),\qquad{x}\in\C(\OO),
\]
where $\omega_{1,1}$ is the functional ${a}\mapsto\is{e_1}{ae_1}$ on $\cK$ with $e_1$ denoting the first vector of the standard basis of $\ell_2$.
\end{lemma}

In what follows, given $\Psi\in\Mor(\C(\OO),\C(\PP)\tens\cK\tens\sB)$ we will denote mappings of the form
\[
x\longmapsto(\id\tens\omega_{1,1}\tens\id)\bigl(\Psi(x)\bigr),\qquad{x}\in\C(\OO)
\]
described in Lemma \ref{aux_ksp} by $\Psi_{1,1}$.

Consider the universal quantum family of maps $\bPhi_{\PP,\OO}\colon\CC(\OO)\to\C(\PP)\tens\C(\MM_{\PP,\OO})$ and let $\SS_{\PP,\OO}$ be the closed linear span of the slices of $\bPhi_{\PP,\OO}$ i.e.
\begin{equation}\label{newRemSPO}
\SS_{\PP,\OO}=\overline{\operatorname{span}}\bigl\{(\omega\tens\id)\bPhi_{\PP,\OO}(x)\,\bigr|\bigl.\,x\in\C(\OO),\:\omega\in\C(\PP)^*\bigr\}.
\end{equation}
Then $\SS_{\PP,\OO}$ is an operator system. Moreover $\SS_{\PP,\OO}$ comes equipped with a map 
\begin{equation}\label{newRembph}
\bph_{\PP,\OO}\colon\C(\OO)\longrightarrow\C(\PP)\tens\SS_{\PP,\OO}
\end{equation}
defined as $\bPhi_{\PP,\OO}$ viewed as a map $\C(\OO)\to\C(\PP)\tens\SS_{\PP,\OO}$. This map is clearly u.c.p. In the next theorem we shall prove that $\SS_{\PP,\OO}$ has an interesting universal property.

\begin{theorem}\label{thm_op_syst}
Let $\PP$ be a finite quantum space and $\OO$ a compact quantum space such that $\C(\OO)$ is separable. Let $\SS_{\PP,\OO}\subset\C(\MM_{\PP,\OO})$ be the operator system equipped with the u.c.p.~map $\bph_{\PP,\OO}\colon\C(\OO)\to\C(\PP)\tens\SS_{\PP,\OO}$ described by \eqref{newRemSPO} and \eqref{newRembph}. For any operator system $\sS$ and any u.c.p.~map $\psi\colon\C(\OO)\to\C(\PP)\tens\sS$ there exists a unique u.c.p.~map $\lambda\colon\SS_{\PP,\OO}\to\sS$ such that the diagram
\[
\xymatrix{
\C(\OO)\ar[rr]^-{\bph_{\PP,\OO}}\ar@{=}[d]&&\C(\PP)\tens\SS_{\PP,\OO}\ar[d]^{\id\tens\lambda}\\
\C(\OO)\ar[rr]^-\psi&&\C(\PP)\tens\sS
}
\]
commutes.
\end{theorem}

\begin{proof}
Let $\sB$ be the \cst-algebra generated by the set
\[
\bigl\{(\phi\tens\id)\bigl(\psi(x)\bigr)\,\bigr|\bigl.\,x\in\C(\OO),\:\phi\in\C(\PP)^*\bigr\}.
\]
Then $\sB$ is separable unital and $\psi$ is a u.c.p.~map $\C(\OO)\to\C(\PP)\tens\sB$. By Lemma \ref{aux_ksp} $\psi=\Psi_{1,1}$ for some $\Psi\in\Mor(\C(\OO),\C(\PP)\tens\cK\tens\sB)$.
The universal property of $(\MM_{\PP,\OO},\bPhi_{\PP,\OO})$ provides $\Lambda\in\Mor(\C(\MM_{\PP,\OO}),\cK\tens\sB)$ such that $\Psi=(\id\tens\Lambda)\comp\bPhi_{\PP,\OO}$. Let $\widetilde{\lambda}=(\omega_{1,1}\tens\id)\comp\Lambda$. Then $\widetilde{\lambda}$ is a u.c.p.~map $\C(\MM_{\PP,\OO})\to\sB$ and defining $\lambda=\bigl.\widetilde{\lambda}\bigr|_{\SS_{\PP,\OO}}$ we obtain
\begin{equation}\label{philambda}
\psi=(\id\tens\lambda)\comp\bph_{\PP,\OO}.
\end{equation}
Slicing both sides of this equation with $\omega\in\C(\PP)^*$ we obtain $\lambda\bigl((\omega\tens\id)\bPhi_{\PP,\OO}(x)\bigr)=(\omega\tens\id)\bigl(\psi(x)\bigr)$ for all $x\in\C(\OO)$. This means, first of all, that the range of $\lambda$ is contained in $\sS$ and secondly, that formula \eqref{philambda} determines the values of $\lambda$ on elements of the form $(\omega\tens\id)\bPhi_{\PP,\OO}(x)$. Since such elements span $\SS_{\PP,\OO}$, this shows that $\lambda$ is unique.
\end{proof}

In what follows whenever $\sS$ is an operator system in a \cst-algebra $\sA$ the symbol $\cst\langle\sS\rangle$ will denote the \cst-subalgebra of $\sA$ generated by $\sS$. The enveloping \cst-algebra (or \cst-envelope) of $\sS$ in the sense of Hamana (\cite[Definition 2.5]{hamana}) will be denoted by $\cst_\env(\sS)$.

\begin{remark}\label{embSS}
Let $\PP,\OO,\SS_{\PP,\OO}$ and $\bph_{\PP,\OO}$ be as above. Then by construction we have
\begin{equation}\label{iota1}
\bPhi_{\PP,\OO}=(\id\tens\iota)\comp\bph_{\PP,\OO},
\end{equation}
where $\iota$ is the inclusion $\SS_{\PP,\OO}\hookrightarrow\C(\MM_{\PP,\OO})$.
\end{remark}

\begin{definition}\label{def:hypersep}
Let $\sA$ be a \cst-algebra and $\sS\subset\sA$ be an operator system. We say that 
\begin{enumerate}
\item the embedding $\sS\subset \sA$ is \emph{hyperrigid} (cf.~\cite[Theorem 2.1]{arvesonII}) if given a Hilbert space $\sH$, a $*$-homomorphism $\pi\colon\sA\to\B(\sH)$ and a u.c.p.~map $\eta\colon\sA\to\B(\sH)$ satisfying $\bigl.\pi\bigr|_{\sS}=\bigl.\eta\bigr|_{\sS}$ we have $\pi=\eta$,
\item\label{def:hypersep2} the embedding $\sS\subset\sA$ has a \emph{strong extension property} if for every \cst-algebra $\sB$ and a u.c.p.~map $\psi\colon\sS\to\sB$ there exists a u.c.p.~extension $\eta\colon\sA\to\sB$ of $\psi$.
\end{enumerate}
\end{definition}

\begin{theorem}\label{HE}
Let $\sS\subset\sA$ be hyperrigid and $\cst\langle\sS\rangle=\sA$. Then $\cst_\env(\sS)=\sA$.
\end{theorem}

\begin{proof}
The universal property of $\cst_\env(\sS)$ (see \cite[Corollary 4.2]{hamana}) provides a surjective $*$-ho\-mo\-mor\-phism $\rho\colon\sA\to\cst_\env(\sS)$ such that $\bigl.\rho\bigr|_\sS$ is the identity on $\sS$. We will prove that $\rho$ is injective. Indeed, embed $\sA$ into $\B(\sH)$ and denote the embedding by $\iota\colon\sA\to\B(\sH)$. Let $\eta\colon\cst(\sS)\to\B(\sH)$ be the u.c.p.~extension of the embedding $\bigl.\iota\bigr|_{\sS}\colon\sS\to\B(\sH)$. Then $\eta\comp\rho\colon\sA\to\B(\sH)$ is a u.c.p.~map such that $\bigl.\eta\comp\rho\bigr|_{\sS}=\bigl.\iota\bigr|_{\sS}$. By hyperrigidity of $\sS$ in $\sA$ we get $\eta\comp\rho=\iota$ and hence $\rho$ is injective.
\end{proof}

\begin{remark}
Theorem \ref{HE} remains true even without the assumption that $\cst\langle\sS\rangle=\sA$: in \cite[Theorem~3.10]{HK2019} it was proven that if $\sS\subset\sA$ is hyperrgid then $\cst_\env(\sS)=\cst\langle\sS\rangle$. This can be deduced from Dritschel-McCullough theory of boundary representations (\cite{DM2005}). Nevertheless we presented a direct proof of Theorem~\ref{HE} under the assumptions we are interested in. 
\end{remark}

The hyperrigidity of $\SS_{\PP,\OO}\subset \C(\MM_{\PP,\OO})$ follows from the following more general fact.

\begin{theorem}\label{thm_hyper}
Suppose that $\sA$ and $\sB$ are unital \cst-algebras, $\sC$ is a \cst-algebra and let $\Phi\in\Mor(\sA,\sC\tens\sB)$. Define the operator system $\sS\subset\sB$
\[
\sS=\overline{\operatorname{span}}\bigl\{(\omega\tens\id)\bigl(\Phi(x)\bigr)\,\bigr|\bigl.\,\omega\in\sC^*,\:x\in\sA\bigr\}.
\]
If $\cst\langle\sS\rangle = \sB$ then $\sS$ is hyperrigid in $\sB$ and $\sB = \cst_\env(\sS)$. In particular for any finite quantum space $\PP$ and any compact quantum space $\OO$ the embedding $\SS_{\PP,\OO}\subset\C(\MM_{\PP,\OO})$ is hyperrigid and $\cst_\env(\sS)=\C(\MM_{\PP,\OO})$. 
\end{theorem}

\begin{proof}
Denoting by $\cU(\sA)$ the set of unitary elements in $\sA$ we easily observe that
\[
\sS=\overline{\operatorname{span}}\bigl\{(\omega\tens\id)\bigl(\Phi(x)\bigr)\,\bigr|\bigl.\,\omega\in\sC^*,\:x\in\cU(\sA)\bigr\}.
\]
Let $\eta\colon\sB\to\B(\sH)$ be a u.c.p.~map, and $\pi\colon\sB\to\B(\sH)$ be a $*$-homomorphism satisfying
\begin{equation}\label{pieta_eq}
\bigl.\eta\bigr|_{\sS}=\bigl.\pi\bigr|_{\sS}.
\end{equation}
In order to prove that $\eta=\pi$ it is enough to prove that $\sS$ is in the multiplicative domain of $\eta$. 

Let $\sL$ be a Hilbert space such that $\sC\subset\B(\sL)$ and let us consider the u.c.p.~map $(\id\tens\eta)\colon\M(\cK(\sL)\tens\sB)\to\M(\cK(\sL)\tens\cK(\sH))$. Equation \eqref{pieta_eq} is equivalent to
\[
(\id\tens\eta)\comp\Phi=(\id\tens\pi)\comp\Phi.
\]
Furthermore, if $u\in\cU(\sA)$ then 
\begin{align*}
\Bigl((\id\tens\eta)\bigl(\Phi(u)\bigr)\Bigr)^*\Bigl((\id\tens\eta)\bigl(\Phi(u)\bigr)\Bigr)
&=(\id\tens\pi)\bigl(\Phi(u^*)\bigr)(\id\tens\pi)\bigl(\Phi(u)\bigr)\\
&=(\id\tens\pi)\bigl(\Phi(u^*u)\bigr)=\I\\
&=(\id\tens\eta)\bigl(\Phi(\I)\bigr)=(\id\tens\eta)\bigl(\Phi(u^*u)\bigr)
\end{align*}
and similarly 
\[
\Bigl((\id\tens\eta)\bigl(\Phi(u)\bigr)\Bigr)\Bigl((\id\tens\eta)\bigl(\Phi(u)\bigr)\Bigr)^*=(\id\tens\eta)\bigl(\Phi(uu^*)\bigr).
\] 
This shows that $\Phi(u)$ is in the multiplicative domain of $(\id\tens\eta)$. By Lemma \ref{mult_dom_arg_K} we conclude that $\sS$ is in the multiplicative domain of $\eta$ which concludes the proof.
\end{proof}

\begin{theorem}\label{strongArveson}
Let $\PP$ be a finite quantum space and let $\OO$ be a compact quantum space such that $\C(\OO)$ is separable. The embedding $\SS_{\PP,\OO}\subset\C(\MM_{\PP,\OO})$ has the strong extension property. 
\end{theorem}

\begin{proof}
Let $\psi\colon\SS_{\PP,\OO}\to\sA$ be a u.c.p.~map. Applying Lemma \ref{aux_ksp} to the composition
\[
(\id\tens\psi)\comp\bph_{\PP,\OO}\colon\C(\OO)\longrightarrow\C(\PP)\tens\sA
\]
we find a $*$-homomorphism $\Psi\colon\C(\OO)\to\C(\PP)\tens\M(\cK\tens\sA)$ such that $\Psi_{1,1}=(\id\tens\psi)\comp\bph_{\PP,\OO}$. By the universal property of $\C(\MM_{\PP,\OO})$ there exists a $*$-homomorphism $\Lambda\colon\C(\MM_{\PP,\OO})\to\M(\cK\tens\sA)$ such that
\[
\Psi=(\id\tens\Lambda)\comp\bPhi_{\PP,\OO}.
\]
The proof is completed by noting that $(\omega_{1,1}\tens\id)\comp\Psi$ is a u.c.p.~map with range in $\sA$ which extends $\psi$.
\end{proof}

Let us recall from \cite[Theorem 6.3]{KPTT} that the \emph{commuting tensor product} of operator systems (defined in \cite[Section 6]{KPTT}) is \emph{functorial} (\cite[Section 3]{KPTT}) which implies that for any operator system $\sS$ there is a canonical u.c.p.~map $\SS_{\PP,\OO}\tens_\cc\sS\to\C(\MM_{\PP,\OO})\tens_\cc\sS$ extending the canonical inclusion on the algebraic tensor products. Furthermore, by \cite[Theorem 6.7]{KPTT} we have $\C(\MM_{\PP,\OO})\tens_\cc\sS=\C(\MM_{\PP,\OO})\tens_\mmax\sS$ (cf.~also \cite[Section 5]{KPTT}). Thus there is a canonical u.c.p.~map
\[
\SS_{\PP,\OO}\tens_\cc\sS\longrightarrow\C(\MM_{\PP,\OO})\tens_\mmax\sS.
\]

\begin{lemma}\label{orderEmb}
Let $\PP$ be a finite quantum space, $\OO$ a compact quantum space such that $\C(\OO)$ is separable and let $\sS$ be an arbitrary operator system. Then the canonical map $\SS_{\PP,\OO}\tens_\cc\sS\to\C(\MM_{\PP,\OO})\tens_\mmax\sS$ is a complete order embedding. 
\end{lemma}

\begin{proof} We employ a modification of the techniques used in the proofs of \cite[Proposition~4.6]{Harris2016} and \cite[Lemma 1.11]{BGH}. Since $\C(\MM_{\PP,\OO})$ is a unital \cst-algebra, by \cite[Theorem 6.6]{KPTT} we can replace the maximal tensor product $\C(\MM_{\PP,\OO})\tens_\mmax\sS$ be the commuting one. Since the embedding $\sS_{\PP,\OO}\hookrightarrow\C(\MM_{\PP,\OO})$ is u.c.p., by the functoriality of the commuting tensor product, \cite[Theorem 6.3]{KPTT}, the map $\SS_{\PP,\OO}\tens_\cc\sS\to\C(\MM_{\PP,\OO})\tens_\cc\sS$ is also u.c.p.

Let $m$ be a natural number and take $y=y^*\in\Mat_m(\SS_{\PP,\OO}\tens_\cc\sS)$ which is positive when viewed as an element of $\Mat_m(\C(\MM_{\PP,\OO})\tens_\cc\sS)$. We have to show that $y$ is positive when viewed as an element of $\Mat_m(\SS_{\PP,\OO}\tens_\cc\sS)$, i.e.~given u.c.p.~maps $\ph_1\colon\SS_{\PP,\OO}\to\B(\sH)$ and $\ph_2\colon\sS\to\B(\sH)$ with commuting ranges we have to prove the positivity of $\bigl(\id\tens(\ph_1\cdot\ph_2)\bigr)(y)$ (cf.~\cite[Section 6]{KPTT}, the definition of $\ph_1\cdot\ph_2$ is on page 289, see also the proof of Proposition \ref{PropReali}).

Denoting by $\sA_{\ph_1}$ the \cst-algebra generated by the range of $\ph_1$ and using Theorem \ref{strongArveson} we find a u.c.p.~extension $\widetilde{\ph}_1\colon\C(\MM_{\PP,\OO})\to\sA_{\ph_1}$ of $\ph_1$. In particular $\widetilde{\ph}_1$ and $\ph_2$ have commuting ranges. This shows that $\widetilde{\ph}_1\cdot\ph_2\colon\C(\MM_{\PP,\OO})\tens_\cc\sS\to\B(\sH)$ is a u.c.p.~map that extends $\ph_1\cdot\ph_2\colon\SS_{\PP,\OO}\tens_\cc\sS\to\B(\sH)$ and we conclude that $\bigl(\id\tens(\ph_1\cdot\ph_2)\bigr)(y)=\bigl(\id\tens(\widetilde{\ph}_1\cdot\ph_2)\bigr)(y)$. Clearly the latter element is positive. 
\end{proof}

\begin{corollary}\label{corOrderEmb}
With $\PP$ and $\OO$ as above the canonical map $\SS_{\PP,\OO}\tens_\cc\SS_{\PP,\OO}\to\C(\MM_{\PP,\OO})\tens_\mmax\C(\MM_{\PP,\OO})$ is a complete order embedding. 
\end{corollary}

\begin{proof}
First, by Lemma \ref{orderEmb} applied to $\sS=\SS_{\PP,\OO}$ we obtain a complete order embedding
\[
\SS_{\PP,\OO}\tens_\cc\SS_{\PP,\OO}\longrightarrow\C(\MM_{\PP,\OO})\tens_\mmax\SS_{\PP,\OO}=\C(\MM_{\PP,\OO})\tens_\cc\SS_{\PP,\OO}
\]
(cf.~the remarks preceding Lemma \ref{orderEmb}). Then, using the fact that the tensor products $\tens_\cc$ and $\tens_\mmax$ are symmetric (\cite[Theorems 5.5 and 6.3]{KPTT}) we apply the Lemma again to obtain a complete order embedding $\C(\MM_{\PP,\OO})\tens_\cc\SS_{\PP,\OO}\to\C(\MM_{\PP,\OO})\tens_\mmax\C(\MM_{\PP,\OO})$.
\end{proof}

\section{Lifting property of \texorpdfstring{$\C(\MM_{\PP,\OO})$}{CMPO} and \texorpdfstring{$\SS_{\PP,\OO}$}{SPO}}\label{sect:LP}

Let $\sA$ be a \cst-algebra. Following \cite[Definition 13.1.1]{BrownOzawa} we say that $\sA$ has the \emph{lifting property} if given any \cst-algebra $\sB$ with an ideal $\sJ\subset\sB$, any contractive completely positive $\ph\colon\sA\to\sB/\sJ$ admits a contractive completely positive lift $\widetilde{\ph}\colon\sA\to\sB$. For unital algebras the lifting property is equivalent to the analogous property with ``contractive completely positive'' replaced by ``unital completely positive'' (\cite[Lemma 13.1.2]{BrownOzawa}).

\begin{theorem}\label{LPprop}
Let $\PP$ be a finite quantum space and let $\OO$ be a compact quantum space such that $\C(\OO)$ is separable and has the lifting property. Then the \cst-algebra $\C(\MM_{\PP,\OO})$ has the lifting property.
\end{theorem}

The famous Choi-Effros lifting theorem (\cite[Theorem C.3]{BrownOzawa}) states that any separable nuclear \cst-algebra has the lifting property. We will see in Remark \ref{rem:nonnuc} that $\C(\MM_{\PP,\OO})$ is almost never nuclear. We need the assumption of separability of $\C(\OO)$ to be able to use Kasparov's dilation theorem.

\begin{proof}[Proof of Theorem \ref{LPprop}]
Let $\sB$ be a unital \cst-algebra with an ideal $\sJ\subset\sB$ and let $\rho\colon\C(\MM_{\PP,\OO})\to\sB/\sJ$ be a unital $*$-homomorphism. Consider the homomorphism
\[
\Psi=(\id\tens\rho)\comp\bPhi_{\PP,\OO}\colon\C(\OO)\longrightarrow\C(\PP)\tens(\sB/\sJ)=\bigl(\C(\PP)\tens\sB\bigr)/\bigl(\C(\PP)\tens\sJ\bigr).
\]
By assumption $\Psi$ admits a u.c.p.~lift $\widetilde{\Psi}\colon\C(\OO)\to\C(\PP)\tens\sB$. Now we use Lemma \ref{aux_ksp} to write $\widetilde{\Psi}$ in the form $\widetilde{\Psi}=\Theta_{1,1}$ for a $\Theta\in\Mor(\C(\OO),\C(\PP)\tens\cK\tens\sB)$.

By the universal property of $(\MM_{\PP,\OO},\bPhi_{\PP,\OO})$ there exists a unique $\Lambda\in\Mor(\C(\MM_{\PP,\OO}),\cK\tens\sB)$ such that $\Theta=(\id\tens\Lambda)\comp\bPhi_{\PP,\OO}$. Now if $q$ denotes the quotient map $\sB\to\sB/\sJ$ then
\begin{align*}
(\id\tens\rho)\comp\bPhi_{\PP,\OO}=\Psi&=(\id\tens{q})\comp\widetilde{\Psi}\\
&=(\id\tens{q})\comp(\id\tens\omega_{1,1}\tens\id)\comp\Theta\\
&=(\id\tens{q})\comp(\id\tens\omega_{1,1}\tens\id)\comp(\id\tens\Lambda)\comp\bPhi_{\PP,\OO},
\end{align*}
so for any $x\in\C(\OO)$ and $\phi\in\C(\PP)^*$
\[
\rho\bigl((\phi\tens\id)\bPhi_{\PP,\OO}(x)\bigr)=\Bigl(q\comp\bigl((\omega_{1,1}\tens\id)\comp\Lambda\bigr)\Bigr)\bigl((\phi\tens\id)\bPhi_{\PP,\OO}(x)\bigr)
\]
In other words for any $y\in\SS_{\PP,\OO}$ we have
\[
\rho(y)=\Bigl(q\comp\bigl((\omega_{1,1}\tens\id)\comp\Lambda\bigr)\Bigr)(y),
\]
and since $\rho\colon\C(\MM_{\PP,\OO})\to\sB/\sJ$ is a $*$-homomorphism, the hyperrigidity of $\SS_{\PP,\OO}\subset\C(\MM_{\PP,\OO})$, (cf.~Theorem \ref{thm_hyper}), implies that $q\comp\bigl((\omega_{1,1}\tens\id)\comp\Lambda\bigr)=\rho$ on the whole \cst-algebra $\C(\MM_{\PP,\OO})$. In particular $(\omega_{1,1}\tens\id)\comp\Lambda$ is a completely positive lift of $\rho$.

Having established that $*$-homomorphisms $\C(\MM_{\PP,\OO})\to\sB/\sJ$ admit u.c.p.~lifts we use Kasparov's theorem again to obtain lifts of u.c.p.~maps: let $\ph\colon\C(\MM_{\PP,\OO})\to\sB/\sJ$ be a u.c.p.~map. Then $\ph=\Phi_{1,1}$, where $\Phi\in\Mor(\C(\MM_{\PP,\OO}),\cK\tens(\sB/\sJ))$. Since $\M(\cK\tens(\sB/\sJ))$ is the image of $\M(\cK\tens\sB)$ under the canonical extension of $\id\tens{q}\colon\cK\tens\sB\to\cK\tens(\sB/\sJ)$ (by \cite[Proposition 6.8]{lance}, owing to the fact that the \cst-algebra $\cK\tens\sB$ is $\sigma$-unital) it has a completely positive unital lift $\widetilde{\Phi}\colon\C(\MM_{\PP,\OO})\to\M(\cK\tens\sB)$, so
\[
\ph=(\omega_{1,1}\tens\id)\comp\Phi=(\omega_{1,1}\tens\id)\comp(\id\tens{q})\comp\widetilde{\Phi}=q\comp\bigl((\omega_{1,1}\tens\id)\comp\widetilde{\Phi}\bigr),
\]
i.e.~$(\omega_{1,1}\tens\id)\comp\widetilde{\Phi}$ is a u.c.p.~lift of $\ph$.
\end{proof}

\begin{remark}\label{rem:nonnuc}
All separable nuclear \cst-algebras have the lifting property (by the Choi-Effros theorem). However the \cst-algebras $\C(\MM_{\PP,\OO})$ are usually not nuclear. For example denoting by $\boldsymbol{2}$ and $\boldsymbol{3}$ the two- and three-point space respectively, by \cite[Corollary II.2]{synchr} we have
$\C(\MM_{\boldsymbol{2},\boldsymbol{3}})\cong\CC^2\fp\CC^2\fp\CC^2\cong\cst(\ZZ_2\fp\ZZ_2\fp\ZZ_2)$ and $\C(\MM_{\boldsymbol{3},\boldsymbol{2}})\cong\CC^3\fp\CC^3\cong\cst(\ZZ_3\fp\ZZ_3)$ which are both non-nuclear (\cite[Theorem 1.1]{paschke}).

In view of statements \eqref{functoriality:1} and \eqref{functoriality:2} of Theorem \ref{thm:functoriality} we immediately find that
\begin{itemize}
\item if $\C(\OO)$ has at least two characters and $\dim{\C(\PP)}\in\{3,5,6.\dotsc\}$ then there is a surjective $\pi\colon\C(\OO)\to\C(\boldsymbol{2})$ and an injective $\rho\colon\C(\boldsymbol{3})\to\C(\PP)$ and hence $\C(\MM_{\PP,\OO})$ surjects onto $\C(\MM_{\boldsymbol{3},\boldsymbol{2}})$,
\item if $\C(\OO)$ has at least three characters and $\dim{\C(\PP)}>1$ then there is a surjective $\pi\colon\C(\OO)\to\C(\boldsymbol{3})$ and an injective $\rho\colon\C(\boldsymbol{2})\to\C(\PP)$ and hence $\C(\MM_{\PP,\OO})$ surjects onto $\C(\MM_{\boldsymbol{2},\boldsymbol{3}})$.
\end{itemize}
Consequently $\C(\MM_{\PP,\OO})$ is not nuclear in either case (\cite[Corollary 2.5]{wassermann}).

Similarly if $\C(\OO)$ has $\Mat_n(\CC)$ with $n\geq{2}$ as a quotient and $\dim{\C(\PP)}>1$ then $\C(\MM_{\PP,\OO})$ surjects onto $\C(\MM_{\boldsymbol{2},\OO})\cong\Mat_n(\CC)\fp\Mat_n(\CC)$ which is non-nuclear by \cite[Proposition 6]{duncan}. In particular if $\OO$ is finite, $\dim{\C(\OO)}>2$ and $\dim{\C(\PP)}>1$ then $\C(\MM_{\PP,\OO})$ is not nuclear.
\end{remark}

Now we can show that the universal operator system $\SS_{\PP,\OO}$ has the analogous version of lifting property:

\begin{theorem}
Let $\PP$ be a finite quantum space and $\OO$ a compact quantum space such that $\C(\OO)$ is separable and has the lifting property. Then the operator system $\SS_{\PP,\OO}$ has the lifting property: given a unital \cst-algebra $\sB$ with an ideal $\sJ\subset\sB$ and a u.c.p.~map $\ph\colon\SS_{\PP,\OO}\to\sB/\sJ$ there exists a u.c.p.~lift $\widetilde{\ph}\colon\SS_{\PP,\OO}\to\sB$ of $\ph$.
\end{theorem}

\begin{proof}
Using the strong extension property of $\SS_{\PP,\OO}\subset\C(\MM_{\PP,\OO})$ (Theorem \ref{strongArveson}) we can find a u.c.p.~extension $\eta\colon\C(\MM_{\PP,\OO})\to\sB/\sJ$. Since $\C(\MM_{\PP,\OO})$ has the lifting property, there is a u.c.p.~lift $\widetilde\eta\colon\C(\MM_{\PP,\OO})\to\sB$ of $\eta$. We define $\widetilde{\ph}=\bigl.\widetilde\eta\bigr|_{\SS_{\PP,\OO}}$.
\end{proof}

\section{Realizable non-signalling correlations}\label{sect:realizable}

In this section we will study a class of correlations on finite quantum spaces which we will call \emph{realizable} -- this terminology will be explained below. When the quantum sets are taken to be classical, this notion reduces to the so-called $\mathrm{qc}$-correlations. Thus in this and following section $\PP$ and $\OO$ are finite quantum spaces playing the role of the question and answer (input and output) sets for Alice and Bob.

We begin with an analog of a positive operator values measure (POVM) on a quantum space.

\begin{definition}\label{def:qPOVM}
Let $\OO$ be a finite quantum space and $\sH$ a Hilbert space. A u.c.p.~map $\C(\OO)\to\B(\sH)$ will be called a \emph{quantum positive operator-valued measure} (quantum POVM) on $\OO$. If $\PP$ is another finite quantum space then a u.c.p.~map $\C(\OO)\to\C(\PP)\tens\B(\sH)$ will be referred to as a \emph{quantum family} of POVMs on $\OO$ indexed by $\PP$.
\end{definition}

\begin{remark}
One can obviously replace $\B(\sH)$ in Definition \ref{def:qPOVM} by a fixed unital \cst-algebra $\sA$. Clearly a quantum family of POVMs on $\OO$ (as in Definition \ref{def:qPOVM}) is at the same time a $\C(\PP)\tens\B(\sH)$-valued quantum POVM on $\OO$.
\end{remark}

In what follows we will use the \emph{leg numbering notation} well known in the theory of quantum groups. Let $\sA,\sB$ and $\sC$ be \cst-algebras. We define
\begin{align*}
\iota_{12}&\in\Mor(\sA\tens\sB,\sA\tens\sB\tens\sC),\\
\iota_{23}&\in\Mor(\sA\tens\sB,\sC\tens\sA\tens\sB),\\
\iota_{13}&\in\Mor(\sA\tens\sB,\sA\tens\sC\tens\sB)
\end{align*}
by
\[
\begin{array}{r@{\;=\;}l}
\iota_{12}(a\tens{b})&a\tens{b}\tens\I,\\
\iota_{23}(a\tens{b})&a\tens\I\tens{b},\\
\iota_{13}(a\tens{b})&\I\tens{a}\tens{b},
\end{array}\qquad\quad{a}\in\sA,\:b\in\sB.
\]
For $1\leq{i<j}\leq{3}$ the image of an element $x\in\sA\tens\sB$ under $\iota_{ij}$ will be denoted by $x_{ij}$. Note that if $\Phi$ is another \cst-algebra and $\Phi\in\Mor(\sD,\sA\tens\sB)$ then the mapping
\[
d\longmapsto\Phi(d)_{ij}
\]
from $D$ to the multiplier algebra of the appropriate tensor product is a morphism of \cst-algebras. We will also use analogous notation for operator systems and unital completely positive maps.\footnote{Note that an operator system contains a unit element, so that all maps involved are well defined.}

\begin{definition}
Let $\OO$ and $\PP$ be finite quantum spaces. A u.c.p.~map
\[
T\colon\C(\OO)\tens\C(\OO)\to\C(\PP)\tens\C(\PP)
\]
will be called a \emph{quantum correlation with quantum set of questions $\PP$ and quantum set of answers $\OO$}. We will abbreviate this to the shorter phrase ``$(\PP,\OO)$-correlation.''
\end{definition}

\begin{definition}
Let $\PP$ and $\OO$ be finite quantum spaces. A $(\PP,\OO)$-correlation $T$ will be called \emph{non-signaling} if 
\begin{equation}\label{non-sig}
T\bigl(\C(\OO)\tens\I_{\C(\OO)}\bigr)\subset\C(\PP)\tens\I_{\C(\PP)},\quad
T\bigl(\I_{\C(\OO)}\tens\C(\OO)\bigr)\subset\I_{\C(\PP)}\tens\C(\PP).
\end{equation}
\end{definition}

\begin{remark}
The definition of a non-signalling correlation was introduced in \cite[Section B]{DuanWinter} in the case $\C(\PP)$ and $\C(\OO)$ are full matrix algebras (however two different sets of questions and two sets of answers are allowed). Our definition is dual to that of \cite{DuanWinter} requiring unitality instead of preservation of the trace. It is clear that specifying $\PP$ and $\OO$ to be classical finite sets reduces condition \eqref{non-sig} to the existence of marginals.
\end{remark}

\begin{proposition}\label{PropReali}
Let $\PP$ and $\OO$ be finite quantum spaces, $\sH$ a Hilbert space and $\omega$ a state on $\B(\sH)$. Let $\ph_1,\ph_2\colon\C(\OO)\to\C(\PP)\tens\B(\sH)$ be quantum families of POVMs on $\OO$ indexed by $\PP$ such that
\begin{equation}\label{eq:comRL}
\ph_1(x)_{13}\ph_2(y)_{23}=\ph_2(y)_{23}\ph_1(x)_{13},\qquad{x,y}\in\C(\OO).
\end{equation}
Then there exists a unique linear map $T\colon\C(\OO)\tens\C(\OO)\to\C(\PP)\tens\C(\PP)$ such that
\[
T(x\tens{y})=(\id\tens\id\tens\omega)\bigl(\ph_1(x)_{13}\ph_2(y)_{23}\bigr),\qquad{x,y}\in\C(\OO)
\]
and $T$ is a non-signalling $(\PP,\OO)$-correlation.
\end{proposition}

\begin{proof}
Since $\C(\OO)$ is finite-dimensional, the tensor product $\C(\OO)\tens\C(\OO)$ is equal to the algebraic tensor product $\C(\OO)\tens_\alg\C(\OO)$ on which we define the linear map
\[
\ph_1\cdot\ph_2\colon\C(\OO)\tens\alg\C(\OO)\ni{x}\tens{y}\longmapsto\ph_1(x)_{13}\ph_2(y)_{23}\in\C(\PP)\tens\C(\PP)\tens\B(\sH).
\]
Since both maps $\ph_1$ and $\ph_2$ are unital, $\ph_1\cdot\ph_2$ is also unital. Furthermore, for any $x,y\in\C(\OO)$ we have
\[
(\ph_1\cdot\ph_2)\bigl((x\tens y)^*(x\tens y)\bigr)=(\ph_1\cdot\ph_2)(x^*x\tens{y^*y})=\ph_1(x^*x)_{13}\ph_2(y^*y)_{23},
\]
and the complete positivity of maps $\ph_1$ and $\ph_2$ implies that $\ph_1\cdot\ph_2$ is positive.

Moreover $\ph_1\cdot\ph_2$ is completely positive, since for $a\in\Mat_m(\C(\OO))$ we have
\[
(\ph_1\cdot\ph_2)_m(a)={\ph_1}_m(a)_{13}{\ph_2}_m(a)_{23}
\]
is a product of two commuting positive elements of $\C(\PP)\tens\C(\PP)\tens\B(\sH)$ (by complete positivity of $\ph_1$ and $\ph_2$). Finally $T=(\id\tens\id\tens\omega)\comp(\ph_1\cdot\ph_2)$ is completely positive as the composition of completely positive maps (cf.~\cite[Section 1.5.4]{wassermann}).

Since both $\phi_1$ and $\phi_2$ are unital, we immediately see that $T(\C(\OO)\tens\I_{\C(\OO)})\subset\C(\PP)\tens\I_{\C(\PP)}$ and $T(\I_{\C(\OO)}\tens\C(\OO))\subset\I_{\C(\PP)}\tens\C(\PP)$, hence $T$ is non-signaling.
\end{proof}

We refer to the condition \eqref{eq:comRL} by saying that $\ph_1$ and $\ph_2$ \emph{commute on the second leg}. Note that using the version of Stinespring theorem for pairs of c.p.~maps with commuting ranges (such as \cite[Theorem 1.6]{wassermann}) we can loosen the assumption that $\C(\OO)$ is finite-dimensional.

\begin{definition}
Let $\PP$ and $\OO$ be finite quantum spaces. A $(\PP,\OO)$-correlation $T$ obtained from two quantum families of POVMs $\ph_1,\ph_2\colon\C(\OO)\to\C(\PP)\tens\B(\sH)$ commuting on the second leg via the construction described in Proposition \ref{PropReali} will be called \emph{realizable} and the triple $(\ph_1,\ph_2,\omega)$ will be referred to as the \emph{realization} of $T$.
\end{definition}

The next theorem provides equivalent descriptions of a class of realizable correlations.

\begin{theorem}\label{ThmRealizable}
Let $T\colon\C(\OO)\tens\C(\OO)\to \C(\PP)\tens \C(\PP)$ be a $(\PP,\OO)$-correlation. Then the following conditions are equivalent:
\begin{enumerate}
\item\label{ThmRealizable1} there exists a Hilbert space $\sH$, a pair of u.c.p.~maps $\ph_1,\ph_2\colon\C(\OO)\to\C(\PP)\tens\B(\sH)$ satisfying
\[
\ph_1(x)_{13}\ph_2(y)_{23}=\ph_2(y)_{23}\ph_1(x)_{13},\qquad{x}\in\C(\OO)
\]
and a norm-one vector $\xi\in\sH$ such that
\[
T(x\tens{y})=(\id\tens \id\tens\omega_\xi)\bigl(\ph_1(x)_{13}\ph_2(y)_{23}\bigr),\qquad{x,y}\in\C(\OO),
\]
\item\label{ThmRealizable2} there exists a Hilbert space $\sH$, a pair of unital $*$-homomorphisms $\Phi_1,\Phi_2\colon\C(\OO)\to\C(\PP)\tens\B(\sH)$ satisfying
\[
\Phi_1(x)_{13}\Phi_2(y)_{23}=\Phi_2(y)_{23}\Phi_1(x)_{13},\qquad{x}\in\C(\OO)
\]
and a norm-one vector $\xi\in\sH$ such that
\[
T(x\tens{y})=(\id\tens\id\tens\omega_\xi)\bigl(\Phi_1(x)_{13}\Phi_2(y)_{23}\bigr),\qquad{x,y}\in\C(\OO),
\]
\item\label{ThmRealizable3} there exists a state $\sigma$ on $\C(\MM_{\PP,\OO})\tens_\mmax\C(\MM_{\PP,\OO})$ such that
\[
T(x\tens{y})=(\id\tens\id\tens\sigma)\bigl(\bPhi_{\PP,\OO}(x)_{13}\bPhi_{\PP,\OO}(y)_{24}\bigr),\qquad{x,y}\in\C(\OO),
\]
\item\label{ThmRealizable4} there exists a state $\sigma$ on $\SS_{\PP,\OO}\tens_\cc\SS_{\PP,\OO}$ such that
\[
T(x\tens{y})=(\id\tens\id\tens\sigma)\bigl(\bph_{\PP,\OO}(x)_{13}\bph_{\PP,\OO}(y)_{24}\bigr),\qquad{x,y}\in\C(\OO).
\]
\end{enumerate}
\end{theorem}

\begin{proof}
Ad \eqref{ThmRealizable2} $\Rightarrow$ \eqref{ThmRealizable3}: For $i=1,2$ let $\Lambda_i\colon\C(\MM_{\PP,\OO})\to\B(\sH)$ be the unique unital $*$-ho\-mo\-mor\-phisms satisfying
\[
\Phi_1=(\id\tens\Lambda_1)\comp\bPhi_{\PP,\OO}\quad\text{and}\quad
\Phi_2=(\id\tens\Lambda_2)\comp\bPhi_{\PP,\OO}.
\]
Since for any $\omega_1,\omega_2\in\C(\PP)^*$ the elements $(\omega_1\tens\id)\bigl(\Phi_1(x)\bigr)$ and $(\omega_2\tens\id)\bigl(\Phi_2(y)\bigr)$ commute, we have
\begin{equation}\label{Lambdaab}
\Lambda_1(a)\Lambda_2(b)=\Lambda_2(b)\Lambda_1(a)
\end{equation}
for all $a,b\in\bigl\{(\omega\tens\id)\bPhi_{\PP,\OO}(x)\,\bigr|\bigl.\,x\in\C(\OO),\:\omega\in\C(\PP)^*\bigr\}$. This set generates the \cst-algebra $\C(\MM_{\PP,\OO})$, so \eqref{Lambdaab} is satisfied for all $a,b\in\C(\MM_{\PP,\OO})$. It follows that there is a unique $\Pi\colon\C(\MM_{\PP,\OO})\tens_\mmax\C(\MM_{\PP,\OO})\to\B(\sH)$ such that
\[
\Pi(a\tens{b})=\Lambda_1(a)\Lambda_2(b),\qquad{a,b}\in\C(\MM_{\PP,\OO}).
\]
Let $\sigma$ be the state on $\C(\MM_{\PP,\OO})\tens_\mmax\C(\MM_{\PP,\OO})$ defined as the composition of $\Pi$ and the vector functional $\omega_\xi$. Clearly for any $x,y\in\C(\OO)$
\begin{align*}
(\id\tens\id\tens\sigma)\bigl(\bPhi_{\PP,\OO}(x)_{13}\bPhi_{\PP,\OO}&(y)_{24}\bigr)
=(\id\tens\id\tens\omega_\xi)(\id\tens\id\tens\Pi)\bigl(\bPhi_{\PP,\OO}(x)_{13}\bPhi_{\PP,\OO}(y)_{24}\bigr)\\
&=(\id\tens\id\tens\omega_\xi)\Bigl(\bigl((\id\tens\Lambda_1)\bPhi_{\PP,\OO}(x)\bigr)_{13}
\bigl((\id\tens\Lambda_2)\bPhi_{\PP,\OO}(y)\bigr)_{23}\Bigr)\\
&=(\id\tens\id\tens\omega_\xi)\bigl(\Phi_1(x)_{13}\Phi_2(y)_{23}\bigr)=T(x\tens{y}).
\end{align*}

Ad \eqref{ThmRealizable3} $\Rightarrow$ \eqref{ThmRealizable2}: Let $T$ be given by the formula in \eqref{ThmRealizable3} for some $\sigma$ and let $(\sH_\sigma,\pi_\sigma,\Omega_\sigma)$ be the GNS triple for $\sigma$. Define $\Lambda_1,\Lambda_2\colon\C(\MM_{\PP,\OO})\to\B(\sH_\sigma)$ by
\[
\begin{array}{r@{\;}l}
\Lambda_1(a)&=\pi_\sigma(a\tens\I),\\
\Lambda_2(a)&=\pi_\sigma(\I\tens{a}),
\end{array}\qquad\quad{a}\in\C(\MM_{\PP,\OO})
\]
and let $\Phi_i=(\id\tens\Lambda_i)\comp\bPhi_{\PP,\OO}$ ($i=1,2$). Since $\Lambda_1(a)\Lambda_2(b)=\pi_\sigma(a\tens{b})$ for all $a,b$, we have
\begin{align*}
(\id\tens\id\tens\omega_{\Omega_\sigma})\bigl(\Phi_1(x)_{13}\Phi_2(y)_{23}\bigr)
&=(\id\tens\id\tens\omega_{\Omega_\sigma})(\id\tens\id\tens\pi_\sigma)\bigl(\bPhi_{\PP,\OO}(x)_{13}\bPhi_{\PP,\OO}(y)_{24}\bigr)\\
&=(\id\tens\id\tens\sigma)\bigl(\bPhi_{\PP,\OO}(x)_{13}\bPhi_{\PP,\OO}(y)_{24}\bigr)=T(x\tens{y})
\end{align*}
for all $x,y\in\C(\OO)$.

Ad \eqref{ThmRealizable3} $\Rightarrow$ \eqref{ThmRealizable4}: By Lemma \ref{corOrderEmb} we can treat $\SS_{\PP,\OO}\tens_\cc\SS_{\PP,\OO}$ as a real ordered subspace of $\C(\MM_{\PP,\OO})\tens_\mmax\C(\MM_{\PP,\OO})$. Thus \eqref{ThmRealizable4} follows from \eqref{ThmRealizable3} by taking restrictions (and Remark \ref{embSS}).

Ad \eqref{ThmRealizable4} $\Rightarrow$ \eqref{ThmRealizable3}: As above, we consider $\SS_{\PP,\OO}\tens_\cc\SS_{\PP,\OO}$ as embedded in $\C(\MM_{\PP,\OO})\tens_\mmax\C(\MM_{\PP,\OO})$. Then any state on $\SS_{\PP,\OO}\tens_\cc\SS_{\PP,\OO}$ extends to a state on $\C(\MM_{\PP,\OO})\tens_\mmax\C(\MM_{\PP,\OO})$ and the implication follows from Remark \ref{embSS}.

Ad \eqref{ThmRealizable1} $\Rightarrow$ \eqref{ThmRealizable4}: By the universal property of $(\SS_{\PP,\OO},\bph_{\PP,\OO})$ we have
\[
\ph_1=(\id\tens\lambda_1)\comp\bph_{\PP,\OO}\quad\text{and}\quad\ph_2=(\id\tens\lambda_2)\comp\bph_{\PP,\OO}
\]
for certain u.c.p.~maps $\lambda_1,\lambda_2\colon\SS_{\PP,\OO}\to\B(\sH)$. Moreover the fact that the ranges of $\ph_1$ and $\ph_2$ commute translates to the fact that the ranges of $\lambda_1$ and $\lambda_2$ commute just like in the proof of the implication \eqref{ThmRealizable2} $\Rightarrow$ \eqref{ThmRealizable3} above. Let $\sigma=\omega_\xi\comp(\lambda_1\cdot\lambda_2)$. Then $\sigma$ is a state on $\SS_{\PP,\OO}\tens_\cc\SS_{\PP,\OO}$ and for any $x,y\in\C(\OO)$
\begin{align*}
(\id\tens\id\tens\sigma)\bigl(\bph_{\PP,\OO}(x)_{13}\bph_{\PP,\OO}(y)_{24}\bigr)
&=(\id\tens\id\tens\omega_\xi)\bigl(\id\tens\id\tens(\lambda_1\cdot\lambda_2)\bigr)\bigl(\bph_{\PP,\OO}(x)_{13}\bph_{\PP,\OO}(y)_{24}\bigr)\\
&=(\id\tens\id\tens\omega_\xi)\bigl(\ph_1(x)_{13}\ph_2(y)_{23}\bigr)=T(x\tens{y}).
\end{align*}

The implication Ad \eqref{ThmRealizable2} $\Rightarrow$ \eqref{ThmRealizable1} is clear.
\end{proof}

\section{Synchronous correlations}\label{sect:synchr}

Let $\PP$ and $\OO$ be finite quantum spaces and let us specify the decompositions
\begin{equation}\label{CPCOdec}
\C(\PP)=\bigoplus_{l=1}^{N_\PP}\Mat_{m_l}(\CC)\quad\text{and}\quad
\C(\OO)=\bigoplus_{k=1}^{N_\OO}\Mat_{n_k}(\CC).
\end{equation}
Let $\bigl\{\prescript{l}{}{f_{st}}\,\bigr|\bigl.\,l=1\dotsc,N_\PP,\:s,t=1,\dotsc,m_l\bigr\}$ and $\bigl\{\prescript{k}{}{e_{ij}}\,\bigr|\bigl.\,k=1\dotsc,N_\OO,\:i,j=1,\dotsc,n_k\bigr\}$ be the corresponding systems of matrix units in $\C(\PP)$ and $\C(\OO)$.

Furthermore for each $l$ and $k$ let $\bigl\{\prescript{l}{}{f_s}\bigr\}_{s=1,\dotsc,m_l}$ and $\bigl\{\prescript{k}{}{e_i}\bigr\}_{i=1,\dotsc,n_k}$ be the standard bases of $\CC^{m_l}$ and $\CC^{n_k}$ respectively. Clearly we have
\[
\prescript{l}{}{f_{st}}=\bigket{\prescript{l}{}{f_{s}}}\bigbra{\prescript{l}{}{f_{t}}}
\quad\text{and}\quad
\prescript{k}{}{e_{ij}}=\bigket{\prescript{k}{}{e_{i}}}\bigbra{\prescript{k}{}{e_{j}}}
\]
for all $k,l,i,j,s,t$.

Given a $(\PP,\OO)$-correlation $T\colon\C(\OO)\tens\C(\OO)\to\C(\PP)\tens\C(\PP)$ and natural numbers
\begin{align*}
k,k'&\in\{1,\dotsc,N_\OO\}\\
l,l'&\in\{1,\dotsc,N_\PP\}\\
i,j&\in\{1,\dotsc, n_k\}\\
i',j'&\in\{1,\dotsc, n_{k'}\}\\
s,t&\in\{1,\dotsc,m_l\}\\
s',t'&\in\{1,\dotsc,m_{l'}\}
\end{align*}
we define 
\begin{equation}\label{XinC}
\prescript{kk'}{ll'}{X}_{(ij),(i'j')}^{(st),(s't')}\in\CC
\end{equation}
by
\begin{equation}\label{defXxx}
T(\prescript{k}{}e_{ij}\tens\prescript{k'}{}e_{i'j'})=\sum_{s,t,s',t',l,l'}\prescript{kk'}{ll'}{X}_{(ij),(i'j')}^{(st),(s't')}\bigl(\prescript{l}{}f_{st}\tens \prescript{l'}{}f_{s't'}\bigr).
\end{equation}

\begin{definition}
Let $\PP$ and $\OO$ be finite quantum spaces with decompositions of $\C(\PP)$ and $\C(\OO)$ as in \eqref{CPCOdec}. A $(\PP,\OO)$-correlation $T\colon\C(\OO)\tens\C(\OO)\to\C(\PP)\tens\C(\PP)$ is called \emph{synchronous} if
\[
\sum_{s,t,i,j,k,l}\tfrac{1}{n_km_l}\prescript{kk}{ll}{X}_{(ij),(ij)}^{(st),(st)}=N_\PP,
\]
where the coefficients $\prescript{kk'}{ll'}{X}_{(ij),(i'j')}^{(st),(s't')}$ are defined by \eqref{defXxx}.
\end{definition}

\begin{remark}
The definition of synchronicity proposed above may seem rather technical, but it is in fact strongly related to the classical notion of synchronicity (as in e.g.~\cite[Section II]{DykemaPaulsen}, \cite[Section 2]{HMPS}). In case the sets $\PP$ and $\OO$ are classical we have $n_k=m_l=1$ for all $k,l$ and there are no ``internal indices'' inside matrix blocks, so the collection of numbers \eqref{XinC} related to $T$ reduces to $\bigl\{\prescript{kk'}{ll'}{X}\,\bigr|\bigl.\,k,k'=1,\dotsc,N_\OO,\:l,l'=1,\dotsc,N_\PP\bigr\}$ and it corresponds to the classical correlation matrix via
\[
\prescript{kk'}{ll'}{X}=p(k,k'|l,l'),\qquad{k,k'}=1,\dotsc,N_\OO,\:l,l'=1,\dotsc,N_\PP.
\]
Note that in terms of a game with strategy $p(\cdot,\cdot|\cdot,\cdot)$ for a fixed $l$ the sum of $p(k,k'|l,l)$ over all $k,k'$ is the probability of giving any pair of answers to the pair of questions $(l,l)$, so $\sum\limits_{k,k'}p(k,k'|l,l)=1$. Thus the condition
\[
\sum\limits_{k,l}p(k,k|l,l)=N_\PP
\]
means that for each $l$ we must have $p(k,k'|l,l)=0$ whenever $k\neq{k'}$.
\end{remark}

The next proposition shows that synchronicity of a correlation can be checked by evaluating it on one particular element and then taking its expectation on the maximally entangled state.

\begin{proposition}\label{prop_synch1}
Let $\PP$ and $\OO$ be finite quantum spaces with decompositions of $\C(\PP)$ and $\C(\OO)$ as in \eqref{CPCOdec} and let $T\colon\C(\OO)\tens\C(\OO)\to\C(\PP)\tens\C(\PP)$ be a $(\PP,\OO)$-correlation. Then $T$ is synchronous if and only if 
\[
\is{\phi}{T\biggl(\sum_{i,j,k}\tfrac{1}{n_k}\prescript{k}{}e_{ij}\tens\prescript{k}{}e_{ij}\biggr)\phi}=1,
\]
where $\phi=\tfrac{1}{\sqrt{N_{\PP}}}\sum\limits_{l}\tfrac{1}{\sqrt{m_l}}\sum\limits_s\bigl(\prescript{l}{}f_s\tens \prescript{l}{}f_s\bigr)\in
\biggl(\bigoplus\limits_{l=1}^{N_\PP}\CC^{m_l}\biggr)\tens\biggl(\bigoplus\limits_{l=1}^{N_\PP}\CC^{m_l}\biggr)$.
\end{proposition}

\begin{proof}
We have
\[
T\biggl(\sum_{i,j,k}\tfrac{1}{n_k}\prescript{k}{}e_{ij}\tens\prescript{k}{}e_{ij}\biggr)
=\sum_{i,j,k,l,l's,s',t,t'}\tfrac{1}{n_k}\prescript{kk}{ll'}{X}_{(ij),(ij)}^{(st),(s't')}\bigl(\prescript{l}{}f_{st}\tens\prescript{l'}{}f_{s't'}\bigr).
\]
Furthermore for any $l,l',s,t,s',t'$
\begin{align*}
\bigis{\phi}{\bigl(\prescript{l}{}f_{st}\tens\prescript{l'}{}f_{s't'}\bigr)\phi}
&=\tfrac{1}{N_{\PP}}\sum_{l_1,l_2}\tfrac{1}{\sqrt{m_{l_1}}}\tfrac{1}{\sqrt{m_{l_2}}}
\sum_{s_1,s_2}\bigis{\prescript{l_1}{}f_{s_1}\tens\prescript{l_1}{}f_{s_1}}{\bigl(\prescript{l}{}f_{st}\tens\prescript{l'}{}f_{s't'}\bigr)\prescript{l_2}{}f_{s_2}\tens \prescript{l_2}{}f_{s_2}}\\
&=\tfrac{1}{N_{\PP}}\sum_{l_1,l_2}\tfrac{1}{\sqrt{m_{l_1}}}\tfrac{1}{\sqrt{m_{l_2}}}
\sum_{s_1,s_2}\delta_{ll_1}\delta_{l'l_1}\delta_{ll_2}\delta_{l'l_2}\delta_{ss_1}\delta_{s's_1}\delta_{ts_2}\delta_{t's_2}\\
&=\tfrac{1}{N_{\PP}}\sum_{l_1,l_2}\tfrac{1}{\sqrt{m_{l_1}}}\tfrac{1}{\sqrt{m_{l_2}}}\delta_{ll_1}\delta_{l'l_1}\delta_{ll_2}\delta_{l'l_2}\delta_{ss'}\delta_{tt'}\\
&=\tfrac{1}{N_{\PP}}\tfrac{1}{m_{l}} \delta_{ll'}\delta_{ss'}\delta_{tt'}.
\end{align*}
Substituting this into the first equality we obtain
\[
\is{\phi}{T\biggl(\sum_{i,j,k}\tfrac{1}{n_k}\prescript{k}{}e_{ij}\tens\prescript{k}{}e_{ij}\biggr)\phi}=\tfrac{1}{N_{\PP}}\sum_{s,t,i,k,l,j}\tfrac{1}{n_km_l}\prescript{kk}{ll}{X}_{(ij),(ij)}^{(st),(st)}.
\]
\end{proof}

We now pass to the construction of certain synchronous correlations via the prescription provided by Theorem \ref{ThmRealizable}. It says that realizable $(\PP,\OO)$-correlation is given by a state on $\C(\MM_{\PP,\OO})\tens\mmax\C(\MM_{\PP,\OO})$. Let $\sigma$ be such a state. The corresponding $T_\sigma$ is defined by
\[
T_\sigma(x\tens{y})=(\id\tens\id\tens\sigma)\bigl(\bPhi_{\PP,\OO}(x)_{13}\bPhi_{\PP,\OO}(y)_{24}\bigr),\qquad{x,y}\in\C(\OO).
\]
We define $\prescript{k}{l}V_{ij}^{st}\in\C(\MM_{\PP,\OO})$ by
\[
\bPhi_{\PP,\OO}(\prescript{k}{}{e_{ij}})=\sum_{s,t,l}\prescript{l}{}f_{st}\tens\prescript{k}{l}V_{ij}^{st},
\qquad\quad
\begin{array}{r@{\;}l}
k&=1,\dotsc,N_\OO,\\
i,j&=1,\dotsc,n_k.
\end{array}
\]
With this notation we have
\[
T_\sigma\bigl(\prescript{k}{}{e_{ij}}\tens\prescript{k'}{}{e_{i'j'}}\bigr)=\sum_{s,t,s',t',l,l'}\sigma\bigl(\prescript{k}{l}{V^{st}_{ij}}\tens\prescript{k'}{l'}{V^{s't'}_{i'j'}}\bigr)
\bigl(\prescript{l}{}{f_{st}}\tens\prescript{l'}{}{f_{s't'}}\bigr)
\]
for all $k,k',i,j,i',j'$.

\begin{lemma}\label{sTau}
Let $\tau$ be a trace on $\C(\MM_{\PP,\OO})$. Then there exists a state $\sigma_\tau$ on $\C(\MM_{\PP,\OO})\tens_\mmax\C(\MM_{\PP,\OO})$ such that
\[
\sigma_\tau\bigl(\prescript{k}{l}{V^{st}_{ij}}\tens\prescript{k'}{l'}{V^{s't'}_{i'j'}}\bigr)=\tau\bigl(\prescript{k}{l}{V^{st}_{ij}}\prescript{k'}{l'}{V^{t's'}_{j'i'}}\bigr)
\]
\end{lemma}

\begin{proof}
Let $(\sH_\tau,\pi_\tau,\Omega_\tau)$ be the GNS triple for $\tau$. Since $\tau$ is a trace, the mapping
\[
\pi_\tau(a)\Omega_\tau\longmapsto\pi_\tau(a^*)\Omega_\tau,\qquad{a}\in\C(\MM_{\PP,\OO})
\]
(defined on a dense subspace of $\sH_\tau$) extends to an anti-unitary operator $J_\tau\colon\sH_\tau\to\sH_\tau$. Therefore $\pi_\tau'\colon\C(\MM_{\PP,\OO})\to\B(\sH_\tau)$ defined by
\[
\pi_\tau'(a)=J_\tau\pi_\tau(a^*)J_\tau,\qquad{a}\in\C(\MM_{\PP,\OO})
\]
is an anti-representation of $\C(\MM_{\PP,\OO})$, or in other words, a representation of $\C(\MM_{\PP,\OO})^\op$. Thus there exists a representation $\Pi$ of $\C(\MM_{\PP,\OO})\tens_\mmax\C(\MM_{\PP,\OO})^\op$ on $\sH$ such that
\[
\Pi(a\tens{b})=\pi_\tau(a)\pi_\tau'(b),\qquad{a,b}\in\C(\MM_{\PP,\OO})
\]
and we can define a state $\sigma_\tau'$ on $\C(\MM_{\PP,\OO})\tens_\mmax\C(\MM_{\PP,\OO})^\op$ by
\[
\sigma_\tau'(x)=\is{\Omega_\tau}{\Pi(x)\Omega_\tau},\qquad{x}\in\C(\MM_{\PP,\OO})\tens_\mmax\C(\MM_{\PP,\OO})^\op.
\]

Now recall from Section \ref{sectOpposite} that $\C(\MM_{\PP,\OO})^\op$ is isomorphic to $\C(\MM_{\PP,\OO})$. One such isomorphism arises from each pair of isomorphisms $\C(\OO)^\op\cong\C(\OO)$ and $\C(\PP)^\op\cong\C(\PP)$ (cf.~the proof of Corollary \ref{CorOp}). Taking the transposition on both these algebras yields an isomorphism $\Upsilon\colon\C(\MM_{\PP,\OO})\to\C(\MM_{\PP,\OO})^\op$ such that
\[
\Upsilon\bigl(\prescript{k}{l}{V^{st}_{ij}}\bigr)=\prescript{k}{l}{V^{ts}_{ji}},\qquad\quad
\begin{array}{r@{\;}l}
k&=1,\dotsc,N_\OO,\\
l&=1,\dotsc,N_\PP,\\
i,j&=1,\dotsc, n_k,\\
s,t&=1,\dotsc,m_l.\\
\end{array}
\]
We define $\sigma_\tau$ by $\sigma_\tau'\comp(\id\tens\Upsilon)$.
\end{proof}

By a mild abuse of notation let us denote the $(\PP,\OO)$-correlation arising from $\sigma_\tau$ (as defined in Lemma \ref{sTau}) by $T_\tau$. Clearly
\[
T_\tau\bigl(\prescript{k}{}{e_{ij}}\tens\prescript{k'}{}{e_{i'j'}}\bigr)=\sum_{s,t,s',t',l,l'}\tau\bigl(\prescript{k}{l}{V^{st}_{ij}}\prescript{k'}{l'}{V^{t's'}_{j'i'}}\bigr)
\bigl(\prescript{l}{}{f_{st}}\tens\prescript{l'}{}{f_{s't'}}\bigr)
\]
for all $k,k',i,j,i',j'$, so that
\begin{equation}\label{Xfromtau}
\prescript{kk'}{ll'}{X}_{(ij),(i'j')}^{(st),(s't')}=\tau\left(\prescript{k}{l}{V^{st}_{ij}}\prescript{k'}{l'}{V^{t's'}_{j'i'}}\right),\qquad\quad
\begin{array}{r@{\;}l}
k,k'&=1,\dotsc,N_\OO,\\
l,l'&=1,\dotsc,N_\PP,\\
i,j,i',j'&=1,\dotsc,n_k,\\
s,t,s',t'&=1,\dotsc,m_l.
\end{array}
\end{equation}

\begin{theorem}\label{Ttausynchr}
Let $\PP$ and $\OO$ be finite quantum spaces and $\tau$ a trace on $\C(\MM_{\PP,\OO})$. Then the $(\PP,\OO)$-correlation $T_\tau$ is synchronous.
\end{theorem}

\begin{proof}
First we compute $\bPhi_{\PP,\OO}\bigl(\prescript{k}{}{e_{ij}}\prescript{k'}{}{e_{j'i'}}\bigr)$ in two ways:
\[
\bPhi_{\PP,\OO}\bigl(\prescript{k}{}{e_{ij}}\bigr)\bPhi_{\PP,\OO}\bigl(\prescript{k'}{}{e_{j'i'}}\bigr)
=\bPhi_{\PP,\OO}\bigl(\prescript{k}{}{e_{ij}}\prescript{k'}{}{e_{j'i'}}\bigr)
=\delta_{kk'}\delta_{jj'}\bPhi_{\PP,\OO}\bigl(\prescript{k}{}{e_{ii'}}\bigr).
\]
The left-hand side is
\begin{align*}
\biggl(\sum_{s,t,l}\prescript{l}{}{f_{st}}\tens\prescript{k}{l}{V^{st}_{ij}}\biggr)
\biggl(\sum_{s',t',l'}\prescript{l'}{}{f_{s't'}}\tens\prescript{k'}{l'}{V^{s't'}_{j'i'}}\biggr)
&=\sum_{s,t,l,s',t',l'}
\prescript{l}{}{f_{st}}\prescript{l'}{}{f_{s't'}}
\tens\prescript{k}{l}{V^{st}_{ij}}\prescript{k'}{l'}{V^{s't'}_{j'i'}}\\
&=\sum_{s,t,l,t'}\prescript{l}{}{f_{st'}}
\tens\prescript{k}{l}{V^{st}_{ij}}\prescript{k'}{l}{V^{tt'}_{j'i'}}\\
&=\sum_{s,t,l,t'}\prescript{l}{}{f_{st}}
\tens\prescript{k}{l}{V^{st'}_{ij}}\prescript{k'}{l}{V^{t't}_{j'i'}}\\
&=\sum_{s,t,l}\prescript{l}{}{f_{st}}
\tens\sum_{t'}\prescript{k}{l}{V^{st'}_{ij}}\prescript{k'}{l}{V^{t't}_{j'i'}}
\end{align*}
while the right-hand side is $\delta_{kk'}\delta_{j,j'}\sum\limits_{s,t,l}\prescript{l}{}{f_{st}}\tens\prescript{k}{l}{V^{st}_{ii'}}
=\sum\limits_{s,t,l}\prescript{l}{}{f_{st}}\tens\delta_{kk'}\delta_{j,j'}\prescript{k}{l}{V^{st}_{ii'}}$, so that
\[
\sum_{t'}\prescript{k}{l}{V^{st'}_{ij}}\prescript{k'}{l}{V^{t't}_{j'i'}}
=\delta_{kk'}\delta_{jj'}\prescript{k}{l}{V^{st}_{ii'}}.
\]
which for $t=s$ reads
\begin{equation}\label{reads}
\sum_{t}\prescript{k}{l}{V^{st}_{ij}}\prescript{k'}{l}{V^{ts}_{j'i'}}
=\delta_{kk'}\delta_{jj'}\prescript{k}{l}{V^{ss}_{ii'}}.
\end{equation}
Using this information and denoting by $\operatorname{Tr}_l$ the standard trace on the $l$-th matrix block of $\C(\PP)$ we compute
\begin{align*}
\sum_{s,t}\prescript{kk'}{ll}{X}_{(ij),(i'j')}^{(st),(st)}
&=\sum_{s,t}\tau\bigl(\prescript{k}{l}V_{ij}^{st}\prescript{k'}{l}V_{j'i'}^{ts}\bigr)\\
&=\sum_{s}\delta_{kk'}\delta_{jj'}\tau\bigl(\prescript{k}{l}V_{ii'}^{ss}\bigr)\\
&=\delta_{kk'}\delta_{jj'}(\operatorname{Tr}_l\tens\tau)\biggl(\sum_{s,t,l'}\prescript{l'}{}{f_{st}}\tens\prescript{k}{l'}{V^{st}_{ii'}}\biggr)\\
&=\delta_{kk'}\delta_{jj'}(\operatorname{Tr}_l\tens\tau)\Bigl(\bPhi_{\PP,\OO}\bigl(\prescript{k}{}{e_{ii'}}\bigr)\Bigr)\\
&=\delta_{kk'}\delta_{jj'}(\operatorname{Tr}_l\tens\tau)\Bigl(\bPhi_{\PP,\OO}\bigl(\prescript{k}{}{e_{ip}}\prescript{k}{}{e_{pi'}}\bigr)\Bigr)\\
&=\delta_{kk'}\delta_{jj'}(\operatorname{Tr}_l\tens\tau)\Bigl(\bPhi_{\PP,\OO}\bigl(\prescript{k}{}{e_{ip}}\bigr)\bPhi_{\PP,\OO}\bigl(\prescript{k}{}{e_{pi'}}\bigr)\Bigr)\\
&=\delta_{kk'}\delta_{jj'}(\operatorname{Tr}_l\tens\tau)\Bigl(\bPhi_{\PP,\OO}\bigl(\prescript{k}{}{e_{pi'}}\bigr)\bPhi_{\PP,\OO}\bigl(\prescript{k}{}{e_{ip}}\bigr)\Bigr)\\
&=\delta_{kk'}\delta_{jj'}\delta_{ii'}(\operatorname{Tr}_l\tens\tau)\Bigl(\bPhi_{\PP,\OO}\bigl(\prescript{k}{}{e_{pp}}\bigr)\Bigr)
\end{align*}
for every $p\in\{1,2,\dotsc, n_k\}$.

Putting $k=k'$, $i=i'$, $j=j'$ and summing over $i,j\in\{1,\dotsc, n_k\}$ we get
\[
\sum_{s,t,i,j}\tfrac{1}{n_k^2}\prescript{kk}{ll}{X}_{(ij),(ij)}^{(st),(st)}
=(\operatorname{Tr}_l\tens\tau)\Bigl(\bPhi_{\PP,\OO}\bigl(\prescript{k}{}{e_{pp}}\bigr)\Bigr).
\]
Next, summing over $p\in\{1,\dotsc,n_k\}$ we obtain
\[
\sum_{s,t,i,j}\tfrac{1}{n_k}\prescript{kk}{ll}{X}_{(ij),(ij)}^{(st),(st)}
=(\operatorname{Tr}_l\tens\tau)\biggl(\sum_{p}\bPhi_{\PP,\OO}\bigl(\prescript{k}{}{e_{pp}}\bigr)\biggr)
\]
and summing over $k$ we get
\[
\sum_{s,t,i,j,k}\tfrac{1}{n_k}\prescript{kk}{ll}{X}_{(ij),(ij)}^{(st),(st)}
=(\operatorname{Tr}_l\tens\tau)\biggl(\sum_{p,k}\bPhi_{\PP,\OO}\bigl(\prescript{k}{}{e_{pp}}\bigr)\biggr)
=(\operatorname{Tr}_l\tens\tau)(\I\tens\I)=m_l,
\]
i.e.
\begin{equation}\label{eqtr}
\sum_{s,t,i,j,k}\tfrac{1}{n_km_l}\prescript{kk}{ll}{X}_{(ij),(ij)}^{(st),(st)}=1
\end{equation}
for every $l$. Summing both sides of \eqref{eqtr} over $l$ we finally arrive at
\begin{equation}\label{eqtr2}
\sum_{s,t,i,j,k,l}\tfrac{1}{n_km_l}\prescript{kk}{ll}{X}_{(ij),(ij)}^{(st),(st)}=N_{\PP}
\end{equation}
\end{proof}

In a completely similar manner as in \cite[Theorem~2.5]{BGH} and \cite[Theorem~5.5]{PSSTW} (cf.~also \cite[Theorem 3.2]{HMPS}) we show below that synchronous realizable $(\PP,\OO)$-correlations arise from tracial states on \cst-algebras generated by operators associated to the maps $\Phi_1$ and $\Phi_2$ in the realization $(\Phi_1,\Phi_2,\omega_\xi)$ of the correlation. Moreover, operators associated to $\Phi_1$ and ones associated to $\Phi_2$ are related to each other, in the same way as Alice's and Bob's operators were in \cite[Theorem~2.5]{BGH}. The proof of the next Theorem is an adapted version of the ones of \cite[Theorem~2.5]{BGH} and \cite[Theorem~5.5]{PSSTW}.

\begin{theorem}\label{thm_synch_games}
Let $\PP$ and $\OO$ be finite quantum spaces and let $T\colon\C(\OO)\tens\C(\OO)\to\C(\PP)\tens\C(\PP)$ be a realizable correlation with realization $(\Phi_1,\Phi_2,\omega_\xi)$, where
\[
\Phi_1,\Phi_2\colon\C(\OO)\longrightarrow\C(\PP)\tens\B(\sH)
\]
are unital $*$-homomorphisms with commuting ranges and $\xi\in\sH$ is a unit vector. Let $\bigl\{\prescript{k}{l}U_{ij}^{st}\bigr\}$ and $\bigl\{\prescript{k}{l}W_{ij}^{st}\bigr\}$ be elements of $\B(\sH)$ defined by
\[
\begin{array}{r@{\;}l}
\Phi_1(\prescript{k}{}e_{ij})&=\displaystyle\sum\limits_{l,s,t}\prescript{l}{}f_{st}\tens\prescript{k}{l}U_{ij}^{st},\\
\Phi_2(\prescript{k}{}e_{ij})&=\displaystyle\sum\limits_{l,s,t}\prescript{l}{}f_{st}\tens\prescript{k}{l}W_{ij}^{st},
\end{array}
\qquad\quad{k}=1,\dotsc,N_\OO,\:i,j=1,\dotsc,n_k.
\]
Assume that $T$ is synchronous. Then
\begin{enumerate}
\item $\prescript{k}{l}W_{ij}^{st}\xi=\bigl(\prescript{k}{l}U_{ij}^{st}\bigr)^*\xi$ for all $k,l,i,j,s,t$,
\item the restriction of the vector state $\omega_\xi$ to the \cst-algebra generated by
\begin{equation}\label{genbyW}
\biggl\{\prescript{k}{l}W_{ij}^{st}\,\biggr|\biggl.\,
\begin{array}{r@{\;}lr@{\;}l}
k&=1,\dotsc,N_\OO,&i,j&=1,\dotsc,n_k,\\
l&=1,\dotsc,N_\PP,&s,t&=1,\dotsc,m_l
\end{array}
\biggr\}
\end{equation}
is a trace,
\item the restriction of the vector state $\omega_\xi$ to the \cst-algebra generated by
\begin{equation}\label{genbyU}
\biggl\{\prescript{k}{l}U_{ij}^{st}\,\biggr|\biggl.\,
\begin{array}{r@{\;}lr@{\;}l}
k&=1,\dotsc,N_\OO,&i,j&=1,\dotsc,n_k,\\
l&=1,\dotsc,N_\PP,&s,t&=1,\dotsc,m_l
\end{array}
\biggr\}
\end{equation}
is a trace.
\end{enumerate} 
Moreover the state $\tau$ on $\C(\MM_{\PP,\OO})$ defined as $\tau=\omega_\xi\comp\Lambda_1$, where $\Lambda_1\colon\C(\MM_{\PP,\OO})\to\B(\sH)$ is the unique unital $*$-homomorphism such that
\[
\Phi_1=(\id\tens\Lambda_1)\comp\bPhi_{\PP,\OO}
\]
is a trace and we have $T=T_\tau$.
\end{theorem}

\begin{proof}
Using the Schwarz inequality in $\sH$ and then in $\RR^{2n_k+2m_l}$ we compute
\begin{equation}\label{inequalities}
\begin{split}
N_{\PP}&=\sum_{s,t,i,j,k,l}\tfrac{1}{n_km_l}\prescript{kk}{ll}{X}_{(ij),(ij)}^{(st),(st)}
=\sum_{s,t,i,j,k,l}\tfrac{1}{n_km_l}\bigis{\xi}{\prescript{k}{l}U_{ij}^{st}\prescript{k}{l}W_{ij}^{st}\xi}\\
&\leq\sum_{s,t,i,j,k,l}\tfrac{1}{n_km_l}\Bigl|\bigis{\xi}{\prescript{k}{l}U_{ij}^{st}\prescript{k}{l}W_{ij}^{st}\xi}\Bigr|\\
&=\sum_{s,t,i,j,k,l}\tfrac{1}{n_km_l}\Bigl|\bigis{\bigl(\prescript{k}{l}U_{ij}^{st}\bigr)^*\xi}{\prescript{k}{l}W_{ij}^{st}\xi}\Bigr|\\
&\leq\sum_{s,t,i,j,k,l}\tfrac{1}{n_km_l}\bigl\|\bigl(\prescript{k}{l}U_{ij}^{st}\bigr)^*\xi\bigr\|\bigl\|\prescript{k}{l}W_{ij}^{st}\xi\bigr\|\\
&\leq\sum_{k,l}\tfrac{1}{n_km_l}\biggl(\sum_{s,t,i,j}\bigl\|\bigl(\prescript{k}{l}U_{ij}^{st}\bigr)^*\xi\bigr\|^2\biggr)^{\frac{1}{2}}
\biggl(\sum_{s,t,i,j}\bigl\|\prescript{k}{l}W_{ij}^{st}\xi\bigr\|^2\biggr)^{\frac{1}{2}}\\
&=\sum_{k,l}\tfrac{1}{n_km_l}\biggl(\sum_{s,t,i,j}\bigis{\xi}{\prescript{k}{l}U_{ij}^{st}\bigl(\prescript{k}{l}U_{ij}^{st}\bigr)^*\xi}\biggr)^{\frac{1}{2}}
\biggl(\sum_{s,t,i,j}\bigis{\xi}{\bigl(\prescript{k}{l}W_{ij}^{st}\bigr)^*\prescript{k}{l}W_{ij}^{st}\xi}\biggr)^{\frac{1}{2}}\\
&=\sum_{k,l}\tfrac{1}{n_km_l}\biggl(\sum_{s,t,i,j}\bigis{\xi}{\prescript{k}{l}U_{ij}^{st}\prescript{k}{l}U_{ji}^{ts}\xi}\biggr)^{\frac{1}{2}}
\biggl(\sum_{s,t,i,j}\bigis{\xi}{\prescript{k}{l}W_{ji}^{ts}\prescript{k}{l}W_{ij}^{st}\xi}\biggr)^{\frac{1}{2}}.
\end{split}
\end{equation}

Next we note that
\begin{subequations}
\begin{align}
\sum_{s,t,i,j,k}\tfrac{1}{n_k}\prescript{k}{l}U_{ij}^{st}\prescript{k}{l}U_{ji}^{ts}=\sum_{i,s,k}\prescript{k}{l}U_{ii}^{ss}=m_l\I_H,\label{last1}\\
\sum_{s,t,i,j,k}\tfrac{1}{n_k}\prescript{k}{l}W_{ji}^{ts}\prescript{k}{l}W_{ij}^{st}=\sum_{j,t,k}\prescript{k}{l}W_{jj}^{tt}=m_l\I_H.\label{last2}
\end{align}
\end{subequations}
Indeed, the first equality in \eqref{last1} follows from the fact that
\[
\sum_{t}\prescript{k}{l}{U^{st}_{ij}}\prescript{k'}{l}{U^{ts}_{j'i'}}
=\delta_{kk'}\delta_{jj'}\prescript{k}{l}{U^{ss}_{ii'}},\qquad\quad
\begin{array}{r@{\;}l}
k,k'&=1,\dotsc,N_\OO,\\
l&=1,\dotsc,N_\PP,\\
j,j',i,i'&=1,\dotsc,n_k,\\
s&=1,\dotsc,m_l.
\end{array}
\]
which is obtained by applying $\Lambda_1$ (described in the statement of the theorem) to both sides of \eqref{reads} and noting that
\begin{equation}\label{Lambdajeden}
\Lambda_1\bigl(\prescript{k}{l}{V^{st}_{ij}}\bigr)=\prescript{k}{l}{U^{st}_{ij}},\qquad\quad
\begin{array}{r@{\;}l}
k&=1,\dotsc,N_\OO,\\
l&=1,\dotsc,N_\PP,\\
i,j&=1,\dotsc,n_k,\\
s,t&=1,\dotsc,m_l.
\end{array}
\end{equation}
The second equality in \eqref{last1} is a consequence of the fact that for every $s$ the expression $\sum\limits_{i,k}\prescript{k}{l}U_{ii}^{ss}$ is the $(s,s)$-component in $\Mat_{m_l}(\B(\sH))=\Mat_{m_l}(\CC)\tens\B(\sH)\subset\C(\PP)\tens\B(\sH)$ of the image of the unit $\I\in\C(\OO)$ under the homomorphism $\Phi_1\colon\C(\OO)\to\C(\PP)\tens\B(H)$. Since $\Phi_1$ is unital, we see that
\[
\sum_{i,k}\prescript{k}{l}U_{ii}^{ss}=\I_H
\]
and summing over $s$ we get \eqref{last1}. Similarly we derive \eqref{last2}.

In particular
\[
\sum_{k,l}\tfrac{1}{n_km_l}\biggl(\sum_{s,t,i,j}\bigis{\xi}{\prescript{k}{l}U_{ij}^{st}\prescript{k}{l}U_{ji}^{ts}\xi}\biggr)^{\frac{1}{2}}
\biggl(\sum_{s,t,i,j}\bigis{\xi}{\prescript{k}{l}W_{ji}^{ts}\prescript{k}{l}W_{ij}^{st}\xi}\biggr)^{\frac{1}{2}}=\sum_{k}\is{\xi}{\xi}=N_\PP
\]
and the inequalities in \eqref{inequalities} are actually equalities. From the fact that the first inequality in \eqref{inequalities} is an equality we conclude that $\prescript{kk}{ll}{X}_{(ij),(ij)}^{(st),(st)}\geq{0}$ for all indices. The equality version on the fifth (in)equality of \eqref{inequalities} shows that \begin{equation}\label{UW1}
\prescript{k}{l}W_{ij}^{st}\xi=\prescript{k}{l}\alpha_{ij}^{st}\bigl(\prescript{k}{l}U_{ij}^{st}\bigr)^*\xi
\end{equation}
for all $s,t,i,j,k,l$ where $\prescript{k}{l}\alpha_{ij}^{st} \in\mathbb{T}$. Using this equation we get 
\[
\begin{split}\prescript{kk}{ll}{X}_{(ij),(ij)}^{(st),(st)} &= \bigis{\xi}{\prescript{k}{l}U_{ij}^{st}\prescript{k}{l}W_{ij}^{st}\xi}\\ &=\prescript{k}{l}\alpha_{ij}^{st} \bigis{\xi}{\prescript{k}{l}U_{ij}^{st}\bigl(\prescript{k}{l}U_{ij}^{st}\bigr)^*\xi}
\end{split} \] and since $\prescript{kk}{ll}{X}_{(ij),(ij)}^{(st),(st)}\geq 0$ we conclude that $\prescript{k}{l}\alpha_{ij}^{st}\in\{0,1\}$ and
in both cases we get
\begin{equation}\label{UW}
\prescript{k}{l}W_{ij}^{st}\xi=\bigl(\prescript{k}{l}U_{ij}^{st}\bigr)^*\xi
\end{equation}
for all $s,t,i,j,k,l$.

Following the method of the proof of \cite[Theorem 2.5]{BGH} we conclude that the vector state $\omega_\xi$ on $\B(\sH)$ is a trace when restricted to the \cst-subalgebra of $\B(\sH)$ generated by \eqref{genbyU} and similarly for on the \cst-subalgebra generated by \eqref{genbyW}. To that end for each $(k,l,i,j,s,t)$ we let the symbol $\bigl(\prescript{k}{l}U_{ij}^{st}\bigr)^{-1}$ denote the adjoint of $\prescript{k}{l}U_{ij}^{st}$ and we consider a word
\[
Z=\bigl(\prescript{k_1}{l_1}U_{i_1j_1}^{s_1t_1}\bigr)^{\alpha_1}\cdots \bigl(\prescript{k_q}{l_q}U_{i_qj_q}^{s_qt_q}\bigr)^{\alpha_q},
\]
in the generators \eqref{genbyU} and their adjoints, where $\alpha_a\in\{1,-1\}$ for $a=1,\dotsc,q$. Since each $\prescript{k}{l}W_{ij}^{st}$ commutes with each $\prescript{k'}{l'}U_{i'j'}^{s't'}$, using \eqref{UW} the vector $Z\xi$ can be rewritten in terms of the generators \eqref{genbyW} and their adjoints as
\[
Z\xi=\bigl(\prescript{k_q}{l_q}W_{i_qj_q}^{s_qt_q}\bigr)^{-\alpha_q}\cdots \bigl(\prescript{k_1}{l_1}W_{i_1j_1}^{s_1t_1}\bigr)^{-\alpha_1}\xi.
\]
Thus, just as in \cite[proof of Theorem 2.5]{BGH} and \cite[Proof of Theorem 5.5]{PSSTW} (cf.~\cite[Theorem 3.2]{HMPS}),
\[
\omega_\xi\bigl(\prescript{k}{l}U_{ij}^{st}Z\bigr)=\omega_\xi(Z\prescript{k}{l}U_{ij}^{st}),\quad \omega_\xi\bigl(\prescript{k}{l}U_{ij}^{st}\prescript{k'}{l'}U_{i'j'}^{s't'}Z\bigr)=\omega_\xi(Z\prescript{k}{l}U_{ij}^{st}\prescript{k'}{l'}U_{i'j'}^{s't'})
\]
for all $k,k',l,l',s,s',t,t',i,i',j,j'$. To complete the proof of the corresponding traciality it suffices now to use the induction on the number of $U$'s under $\omega_\xi$. 

In order to prove the final statement of the theorem let us note that by \eqref{Lambdajeden} the range of $\Lambda_1$ is the \cst-algebra generated by the $\prescript{k}{l}U_{ij}^{st}$ (with all possible indices) on which the vector functional $\omega_\xi$ is a trace. It follows that $\tau$ is a trace.

Moreover the Hilbert space $\sH_\xi$ defined as the closure of $\Lambda_1(\C(\MM_{\PP,\OO}))\xi$ can be identified with the GNS space for $\tau$. Using this identification we have
\begin{align*}
\tau\bigl(\prescript{k}{l}V_{ij}^{st}\prescript{k'}{l'}V_{j'i'}^{t's'}\bigr)&=\bigis{\xi}{\prescript{k}{l}U_{ij}^{st}\prescript{k'}{l'}U_{j'i'}^{t's'}\xi}\\
&=\bigis{\xi}{\prescript{k}{l}U_{ij}^{st}\bigl(\prescript{k'}{l'}U_{i'j'}^{s't'}\bigr)^*\xi}\\
&=\bigis{\xi}{\prescript{k}{l}U_{ij}^{st}\prescript{k'}{l'}W_{i'j'}^{s't'}\xi},
\end{align*}
where in the third equality we used \eqref{UW}. In view of \eqref{Xfromtau}, this proves the claim.
\end{proof}

\begin{remark}
Let us give a separate proof that a $(\PP,\OO)$-correlation $T_\tau$ defined by a trace $\tau$ on $\C(\MM_{\PP,\OO})$ as explained before Theorem \ref{Ttausynchr} is synchronous. We will use Proposition \ref{prop_synch1}.

We compute 
\begin{align*}
\is{\phi}{T\biggl(\sum_{i,j,k}\tfrac{1}{n_k}\prescript{k}{}e_{ij}\tens\prescript{k}{}e_{ij}\biggr)\phi}
&=\is{\phi}{\biggl(\sum_{i,j,k,l,l',s,t,s',t'}\tau\bigl(\prescript{k}{l}V_{ij}^{st}\prescript{k}{l'}V_{ji}^{t's'}\bigr)\tfrac{1}{n_k}\bigl(\prescript{l}{}f_{st}\tens\prescript{l'}{}f_{s't'}\bigr)\biggr)\phi}\\
&=\sum_{i,j,k,l,l',s,t,s',t'}\tfrac{1}{n_k}\tau\bigl(\prescript{k}{l}V_{ij}^{st}\prescript{k}{l'}V_{ji}^{t's'}\bigr)\bigis{\phi}{\bigl(\prescript{l}{}f_{st} \tens\prescript{l'}{}f_{s't'}\bigr)\phi}.
\end{align*}
From the identity
\begin{align*}
\bigis{\phi}{\bigl(\prescript{l}{}f_{st}\tens\prescript{l'}{}f_{s't'}\bigr)\phi}&
=\tfrac{1}{N_\PP{m_l}}\delta_{ll'}\sum_{s_1,t_1}\bigis{\prescript{l}{}f_{s_1}}{\prescript{l}{}f_{s}}\bigis{\prescript{l}{}f_{s_1}}{\prescript{l}{}f_{s'}}\bigis{\prescript{l}{}f_{t}}{\prescript{l}{}f_{t_1}}\bigis{\prescript{l}{}f_{t'}}{\prescript{l}{}f_{t_1}}\\
&=\tfrac{1}{N_\PP{m_l}}\delta_{ll'}\delta_{ss'}\delta_{tt'}
\end{align*}
we get
\begin{align*}
\is{\phi}{T\biggl(\sum_{i,j,k}\tfrac{1}{n_k}\prescript{k}{}e_{ij}\tens\prescript{k}{}e_{ij}\biggr)\phi}
&=\sum_{i,j,k,l,l',s,t,s',t'}\tfrac{1}{N_{\PP}m_ln_k}\tau\bigl(\prescript{k}{l}V_{ij}^{st}\prescript{k}{l'}V_{ji}^{t's'}\bigr)\delta_{ll'}\delta_{ss'}\delta_{tt'}\\
&=\sum_{i,j,k,l,s,t}\tfrac{1}{N_{\PP}m_ln_k}\tau\bigl(\prescript{k}{l}V_{ij}^{st}\prescript{k}{l}V_{ji}^{ts}\bigr)\\
&=\sum_{i,j,k,l,s}\tfrac{1}{N_{\PP}m_ln_k}\tau\bigl(\prescript{k}{l}V_{ii}^{ss}\bigr)\\
&=\sum_{i,k,l,s}\tfrac{1}{N_{\PP}m_l}\tau\bigl(\prescript{k}{l}V_{ii}^{ss}\bigr)\\
&=\sum_{l,s}\tfrac{1}{N_{\PP}m_l}\tau(\I)\\
&=\sum_{l,s}\tfrac{1}{N_{\PP}m_l}=1.
\end{align*}

Let us note in the reasoning above we never used the fact that $\tau$ was a trace (although it is needed to define $T_\tau=T_{\sigma_\tau}$, cf.~Lemma \ref{sTau}). In particular, Theorem \ref{thm_synch_games} shows that if $\omega$ is a state on $\C(\MM_{\PP,\OO})$ and there exists a $(\PP,\OO)$-correlation $T$ satisfying 
\[
T\bigl(\prescript{k}{}e_{ij}\tens\prescript{k}{}e_{ij}\bigr)=\sum_{l,l's,s',t,t'}\omega\bigl(\prescript{k}{l}V_{ij}^{st}\prescript{k}{l'}V_{ji}^{t's'}\bigr)
\bigl(\prescript{l}{}f_{st}\tens\prescript{l'}{}f_{s't'}\bigr)
\]
for all $k\in\{1,\dotsc,N_{\OO}\}$ and $i,j\in\{1,\dotsc,n_k\}$ then, if $T$ is realizable, there exists a trace $\tau$ on $\C(\MM_{\PP,\OO})$ such that
\[
\tau\bigl(\prescript{k}{l}V_{ij}^{st}\prescript{k}{l'}V_{ji}^{t's'}\bigr)=\omega\bigl(\prescript{k}{l}V_{ij}^{st}\prescript{k}{l'}V_{ji}^{t's'}\bigr)
\]
for all indices.
\end{remark}

\section*{Acknowledgments}
 
The first author acknowledges the hospitality of the Department of Mathematics of the Indiana University Bloomington during the Fulbright Junior Research Award scholarship funded by the Polish-US Fulbright Commission. The third author was partially supported by the Polish National Agency for the Academic Exchange, Polonium grant PPN/BIL/2018/1/00197 as well as by the FWO–PAS project VS02619N: von Neumann algebras arising from quantum symmetries. The authors also wish to thank Hari Bercovici, Michael Brannan, Samuel Harris, Ching Wei Ho, Vern Paulsen, Adam Skalski, William Slofstra, Thomas Sinclair and Ivan Todorov for illuminating discussions on the subject of the paper.

\end{document}